\documentclass{article}
\usepackage[a4paper]{geometry}
\usepackage[utf8]{inputenc}

\usepackage[colorlinks]{hyperref}
\usepackage[singlelinecheck=false]{caption}
\usepackage{booktabs}

\usepackage{graphicx}
\usepackage{latexsym}

\usepackage{lipsum}
\usepackage{graphicx}
\usepackage{authblk}
\usepackage[english]{babel}
\usepackage{color}

\usepackage{hyperref}
\hypersetup{
    colorlinks=true,
    linkcolor=blue,
    filecolor=blue,      
    urlcolor=blue,
    citecolor=blue,
    }
\usepackage{dsfont}
\usepackage{geometry}
\usepackage{float}
\usepackage{bbm}
\usepackage{enumerate}

\usepackage[nottoc]{tocbibind}
 \usepackage{indentfirst}

\graphicspath{{./figure/}}

\usepackage{pgfplots}
\usepackage{comment}
\usepackage{fontenc}
\usepackage{booktabs}
\usepackage{tabularx}
\usepackage{tikz}
\usepackage{caption}
\usepackage{subcaption}

\DeclareCaptionLabelFormat{custom}
{%
      \textbf{Figure #2}
}
\DeclareCaptionLabelSeparator{custom}{ }

\DeclareCaptionFormat{custom}
{%
\centering
    #1 #2 #3 
}
 \usepackage{tikz-cd}

\usepackage{amsmath,amsfonts,amssymb,amsthm}

\numberwithin{equation}{section}
\usepackage{etoolbox}
\apptocmd{\thebibliography}{}{}{}
\renewcommand*{\Re}{\mathop{}\!\mathrm{Re}\,} 

\theoremstyle{plain}
\newtheorem{theorem}{Theorem}[section]
\newtheorem{lemma}[theorem]{Lemma}

\newtheorem{proposition}[theorem]{Proposition}

\newtheorem{reminder*}{[theorem]Reminder}

\newtheorem{details*}[theorem]{Details}

\newtheorem{comm*}{Comment}
\newtheorem{definition}[theorem]{Definition} 
\newtheorem{definition*}{[theorem]Definition}

\newtheorem{notation*}{Notation}

\newtheorem{remark}[theorem]{Remark}

\title{Critical lack of equilibrium in stochastic kinetic proofreading}
  \author{E. Franco, J. J. L. Velázquez}

\begin{document}

\maketitle

\begin{abstract}
    In this paper we study the relation between the property of detailed balance and the ability of discriminating between different ligands for a class of stochastic models of kinetic proofreading. 
    We prove the existence of a critical amount of lack of detailed balance that the kinetic proofreading models must have in order to have strong specificity for a value of the binding energy $\sigma$. We also prove that the fact that a kinetic proofreading model has a lack of detailed balance that is larger than the critical one does not necessarily yield strong discrimination properties. 
    Indeed, there exist different sets of chemical rates, leading to the same amount of lack of detailed balance, that have strong discrimination property in some cases and not in others. 
\end{abstract}

\tableofcontents

\section{Introduction}
Kinetic proofreading models are chemical reaction networks that can distinguish between ligands leading to a correct output and ligands leading to a wrong output.
A family of kinetic proofreading models were introduced by Hopfield and Ninio in the 70's (see \cite{hopfield1974kinetic,Ninio}). Generalizations of this model, including inhibition effects,  feedbacks and rebindings, have been studied by several authors (see for instance \cite{chan2004feedback,dushek2009role,franccois2016case,franccois2013phenotypic,govern2010fast,rendall2017multiple,sontag2001structure}).

It was noticed in \cite{hopfield1974kinetic} that the kinetic proofreading networks must work in out of equilibrium conditions in order to perform their function. The kinetic proofreading model analysed in \cite{hopfield1974kinetic} is a deterministic linear chemical reaction network modelled by a system of ODEs. Computing the flux solutions of this system, it is possible to observe that the specificity to detect ligands can be much higher in systems that are out of equilibrium compared to the optimal specificity that can be achieved in equilibrium conditions. 
The relation between the dissipation of free energy and the specificity properties of ODEs models of kinetic proofreading systems has been analysed in several papers in the biophysics literature (for instance in \cite{bennett1979dissipation,pineros2020kinetic,wegscheider1902simultane,yu2022energy}).

In this paper we study a probabilistic version of the classical deterministic linear model of kinetic proofreading proposed by Hopfield. 
The ligands are assumed to be characterized by their binding energy $\sigma$ with the receptor.
We assume that when ligands bind to a receptor they form a complex that can be at different phosphorylation states. In particular we assume that $N >1$ is the number of states. Once the complex is formed the ligand can detach from the receptor, with a rate that depends on its binding energy $\sigma$ with the receptor, or it can either gain a phosphate group (i.e. a phosphorylation event takes place), or it loses an inorganic phosphate group (i.e. a dephosphorylation event takes place). 
If the complex ligand receptor reaches the phosphorylation state $N$, then it produces an output. 
The chemical reaction network that we consider can be summarized as follows
\begin{center}
\begin{tikzcd}
&   & &  \ \ \ \ \ S \arrow[dll,swap]\arrow[dl,swap]\arrow[drr,swap]\arrow[dr,swap] \ \ \ \ \ &   &   \\
& C_0 \arrow[urr,swap]  \arrow[r, swap] &  C_1  \arrow[ur,  swap] \arrow[r, swap] \arrow[l, swap]  & \ldots \ldots \arrow[r, swap] \arrow[l, swap]   & C_{N-1} \arrow[ul] \arrow[l, swap]  \arrow[r, swap] \arrow[l, swap]   & C_{N} \arrow[ull, swap  ] \arrow[r, swap] &  output
\end{tikzcd}
\end{center}
here $S$ represents the ligands and $\{ C_k \}_{k=0}^N $ are the complexes ligand receptor at different phosphorylation states. 

In this paper we are interested in applications of the kinetic proofreading mechanism to recognition processes of the immune system. 
There is experimental evidence that low densities of foreign antigens are able to trigger an immune response (e.g. \cite{altan2005modeling}). Therefore, in this paper we consider a stochastic version of the classical deterministic kinetic proofreading model that allows to study the probability that one ligand induces a response of the immune system. 
Other stochastic models of kinetic proofreading have been studied in \cite{currie2012stochastic,kirby2024leisure,kirby2023proofreading, munsky2009specificity,xiao2024leisure}.
However, the question that we address in this paper, i.e. whether strong specificity can take place in equilibrium systems (or in systems that are close to equilibrium in a suitable sense that will be clarified later) is not studied in these papers. 

In the context of immune recognition processes, a ligand is a complex, pMHC, made of a peptide attached on a major histocompatibility complex (MHC). The receptors are the receptors of the $T$-cells and the output corresponds to the response of the immune system (we refer to \cite{janeway2001immunobiology} for details on the role of kinetic proofreading mechanisms in the immune system).
The goal of the kinetic proofreading mechanism is to discriminate between self-ligands, that should not induce an immune response, and foreign-ligands, corresponding for instance to pathogens, that are expected to produce an immune response. 

The chemical reaction network that we consider in this paper is a generalization of the model that we analysed in \cite{franco2025stochastic}. 
In that paper we assume that most of the chemical reactions are one-directional. In particular we assume that dephosphorylation does not occur and that the attachment of a ligand with a receptor always forms a dephosphorylated complex.
In the model that we study in this paper, instead, all the reactions, except for the reaction which produces the output, are assumed to  be bidirectional. 
This generalization allows us to study the role of the lack of detailed balance for the linear kinetic proofreading networks proposed by Hopfield. In Subsection \ref{sec:Delta to infty} we will also show that the model studied in \cite{franco2025stochastic} can be obtained as a limit case of the models studied in this paper.

The property of detailed balance in chemical reaction networks states that at steady state every chemical reaction is balanced by its reverse reaction. 
In particular this implies that, at steady state, there are no fluxes of chemicals in the system. 
We stress that the property of detailed balance must hold in every chemical reaction network at constant temperature that does not exchange substances with the environment. 
However, many biological systems, exchange energy with the environment and can be modelled by effective systems for which the property of detailed balance fails (see \cite{franco2025detailed}). 

The model of kinetic proofreading that we study in this paper does not satisfy the detailed balance property as it is an effective model that describes a chemical network that is in contact with reservoirs of ATP, ADP and phosphate groups that we denote with $P$, which are out of equilibrium. 
In Section \ref{sec:completion} we introduce a model of kinetic proofreading in which we explicitly take into account of the role of ATP, ADP and phosphate groups in the chemical reactions. This model satisfies the property of detailed balance. We then show that if we assume that the concentration of ATP, ADP and phosphate groups is kept at constant non equilibrium values by means of an active mechanism, then we recover the kinetic proofreading model without detailed balance that we study in this paper. 
Notice that here we do not try to model the active mechanisms that keep the concentration of ATP, ADP and phosphate groups out of equilibrium. 

In the previous paper \cite{franco2025stochastic} we studied a class of kinetic proofreading models, simpler than the ones considered in this paper, for which we prove that they have strong discrimination properties (strong specificity).
In this paper we say that a kinetic proofreading network has strong discrimination properties if it is able to distinguish with low error rate between ligands with differences of binding energies that are of order $1/N$, where we recall that $N $ is the number of kinetic proofreading steps (i.e. the number of phosphorylation states). 
One of the goals of this paper is to prove that a minimal amount of lack of detailed balance is necessary in order to obtain this strong discrimination property. We consider a class of linear kinetic proofreading networks of the same type as the network introduced in \cite{hopfield1974kinetic}.
We prove that there exists a critical lack of detailed balance that the system must have in order to have strong specificity.

We now explain how we measure in a quantitative manner the lack of detailed balance. 
It is well known that the fact that the detailed balance property holds or not depends on the chemical rates of the reactions that belong to the cycles of the network. Indeed the Wegscheider criterium (see \cite{wegscheider1902simultane}) guarantees that the property of detailed balance holds in a chemical reaction network if and only if 
\begin{equation} \label{weg condition}
\prod_{ R \in \omega } \frac{ K_R }{ K_{- R }} =1, 
\end{equation}
for every cycle $\omega$ in the chemical network. 
Here a cycle is a set of chemical reactions that when applied to a vector of concentrations has a null net effect. 
In \eqref{weg condition} we are denoting with $K_R $ the rate of the reaction $R$ and with $K_{-R }$ the rate of $-R$, the reverse of the reaction $R$.

Notice that the set of the cycles of a  chemical reaction network has the structure of Abelian group. For the specific  kinetic proofreading network that we study in this paper it is possible to define a basis of cycles that contains three reactions. Indeed we consider the basis  of cycles $\{ \omega_k \}_{k=1}^N$, where for every $k $ we have that the cycle $\omega _k=[R_1^{(k)}, R_2^{(k)}, R_3^{(k)}]$ has the form 
\[
S \overset{R_1^{(k)}} \longrightarrow C_k  \overset{R_2^{(k)}}\longrightarrow C_{k+1} \overset{R_3^{(k)}}\longrightarrow S.
\]
A convenient way to measure the lack of detailed balance of the kinetic proofreading network is to introduce the parameter $\Delta_k$, associated to each cycle $\omega_k$,  defined as 
\begin{equation} \label{delta in cylces}
e^{\Delta_{k} } := \frac{ K_{R_1^{(k)}}K_{R_2^{(k)}} K_{R_3^{(k)}}  }{ K_{- R_1^{(k)} }K_{- R_2^{(k)} }K_{- R_3^{(k)} }}. 
\end{equation}
Notice that if $\Delta_{k} =0$ for every $k=1, \dots , N $, then the detailed balance property holds. We assume that every cycle has the same lack of detailed balance, i.e. for every $k $ it holds that $\Delta_k = \Delta $. 
We will use the parameter $\Delta$ as a measure of the lack of detailed balance of the kinetic proofreading model. 

We prove that the class of kinetic proofreading models studied in this paper can exhibit strong discrimination properties only if $\Delta > \Delta_c$ where $\Delta_c$ depends on the parameters of the model, i.e. on the  phosphorylation rates, the energies of the substances in the network and the binding energy $\sigma$ of the ligand with the receptor.
Hence, the fact that the chemical rates satisfy the condition $\Delta > \Delta_c$ is a necessary condition in order to have strong discrimination properties. 
Since the value of $\Delta$ measures the lack of detailed balance (lack of equilibrium) one could naively think that this value could characterize the behaviour of the system. In other words one could think that the condition $\Delta > \Delta_c $ is sufficient to have strong discrimination properties. It turns out that this is not the case. In Section \ref{sec:alterantive models no db and no disc} we show that some choices of chemical rates $\{ K_{R_j }  \} $ yielding lack of detailed balance, and specifically satisfying $\Delta > \Delta_c $, produce strong discrimination properties while other choices of chemical rates, that give the same value of $\Delta$, do not.  
Therefore, the fact that the network does not satisfy the detailed balance property, is not a sufficient condition for strong specificity, even if $\Delta $ is very large.  Indeed, it turns out that the details of the chemical reactions rates are also crucial.

We already anticipated that we say that a chemical reaction network has strong discrimination properties if it can distinguish with low error rates ligands that are characterized by binding energies whose difference is of order $1/N $. 
We clarify now what we mean with that in this paper. 
 We prove that the characteristic time at which a complex ligand-receptor reaches the state $C_{N}$ and therefore produces a response scales like 
 \[
 T_N \sim \exp\left( \lambda_\Delta  (\sigma, E) N \right) \text{ as } N \to \infty, 
 \]
 where $\lambda_\Delta (\sigma, E) $ is a function of $\sigma, E $ and $\Delta $. 
Suppose that the ligands have a characteristic life time that satisfies 
\begin{equation}\label{tlife}
 T_{life} \sim \exp\left( b N \right)  \text{ as } N \to \infty, 
\end{equation}
for a positive $b$. 
Then, the probability $p_{res}(\sigma)$  that one ligand produces a response satisfies
\[
p_{res} (\sigma) \sim \frac{e^{-\lambda_\Delta (\sigma, E) N}}{e^{-\lambda_\Delta (\sigma, E) N } + e^{-b N  }} \text{ as } N \to \infty 
\] 
 and will be close to $1 $ when $ b - \lambda_\Delta (\sigma, E)  \gtrsim 1/N $ and close to $0$ when  $  \lambda_\Delta (\sigma, E) - b  \gtrsim 1/N $.

In this paper we prove that for any binding energy $\sigma_c \in \mathbb R$ there exists a critical amount of detailed balance $\Delta_c(\sigma_c) $ that has the following property. If $\Delta > \Delta_c(\sigma_c) $ then we can find a parameter $b \in (0, E) $ such that $
\lambda_\Delta (\sigma_c, E) =b< E. $
Around the parameter $\sigma_c$ we have an abrupt transition from response to non-response. 
Indeed, if $\sigma - \sigma_c \gtrsim \frac{1}{N} $ we have that $p_{res} (\sigma) \approx 0 $ while if $ \sigma_c -  \sigma  \gtrsim \frac{1}{N} $, then  $p_{res} ( \sigma) \approx 1 $.  
The region of transition from response to non-response, where we have that $p_{res} ( \sigma ) \in (0,1) $, is of order $1/N$ and is the region where we have 
\[
| \sigma - \sigma_c | \lesssim \frac{1}{N} . 
\] 
 Hence the accuracy of the discrimination is of order $1/N$.
 If instead we have that $\Delta < \Delta_c (\sigma_c) $ it is not possible to find a $b $ such that $\lambda_\Delta(\sigma_c , E) = b $. 
 Hence, in this case, we do not have a sharp transition from non response to response around a critical binding energy separating the energies leading to a response from the energies that do not lead to a response. In this case the network does not have strong discrimination properties.

Notice that in this paper we assume that $T_{life} $ satisfies \eqref{tlife}.
However, also if $ T_{life} \lesssim N  $ it would be possible to study the discrimination properties of the kinetic proofreading model.
Also in this case we would have some discrimination properties, meaning that higher $\sigma $ implies a smaller probability of response. However in that setting we would not obtain the existence of a sharp transition from $1$ to $0$ for the probability of response around a critical value of $\sigma$. 

An interpretation of the critical behaviour found in this paper is the following. In the subcritical regime the response is achieved by means of a direct jump from the signal $S$ to the final state $N$. Due to this fact, there is no dependence on $\sigma $ in the  probability of response as $N \to \infty $. Instead in the supercritical regime the response is achieved through a transition of the system along the phosphorylation chain. In the second case the process takes place faster due to the fact that $\lambda_\Delta(\sigma, E) $ is smaller than $E$. Notice that the direct transition from $S$ to $N$ is very unlikely and is expected to take place in time scales of order $e^{EN}$. However, in the supercritical regime the process is faster and takes place at the time scale $e^{\lambda_\Delta(\sigma, E) N }$, that is much shorter than $e^{EN}$ for $N $ large.

\subsection{Notation}
In this paper we use the notation $\mathbb R_* = (0, \infty ) $ and $\mathbb N_0 = \mathbb N \setminus \{ 0\} $. 
Moreover we use the notation $ f \sim  g $ as $N \to \infty $ to indicate that the two functions $f$ and  $g $ are asymptotically equivalent, i.e. $\lim_{N \to \infty} \frac{f(N)}{g(N)} =1 $. We use the notation $f \approx g $ to indicate that $f$ and $g$ are roughly of the same order. We say that $f \gg g $ as $N \to \infty $ if it holds that $\lim_{N \to \infty } \frac{g(N)}{f(N)}=0$. 
Finally we use $ f \gtrsim g $ to say that there exists a positive constant $c>0$ such that $f \geq c g $. 

\section{Main result of the paper}
In this section we present the model that we study in this paper and the main results that we prove. 
\subsection{The model} \label{sec:model}
We now explain the assumptions on the chemical rates that we make in this paper. 
The particular choice of parameters that we consider allows us to define a chemical reaction network in which $\Delta \geq 0$ measures the lack of detailed balance. 
To define the chemical reaction rates we start associating to each substance in the network an energy. Hence we define a vector of energies $ \overline E=(E_S, E_0, E_1, \dots, E_{N-1}, E_N) \in \mathbb R^{N+2} $.
We assume that the attachment of a ligand to the receptor, i.e. the reactions of the form 
\[
R_1^{(k)} : S \rightarrow  C_k
\] 
take place at rate $K_{R_1^{(k)}} =e^{- E_k } e^{\sigma}$ for every $k \in \{ 0, \dots , N \} $. 
The detachment rate depends on the binding energy between the ligand and the receptor, i.e. the reactions of the form
\[
-R_1^{(k)}:  C_{k} \rightarrow S
\] 
take place at rate $K_{- R_1^{(k)}} = e^{\sigma} $ for every $k \in \{0, \dots, N \} $ where $\sigma \in \mathbb R$. 
The reactions
\[ 
R^{(k)}_2: C_k \rightarrow C_{k+1} \]
take place at rate $K_{R_2^{(k)}}= \alpha e^{\Delta}$ for every $k \in \{ 0, \dots , N \} $. 
On the other hand we assume that the reaction 
\[
-R^{(k)}_2:
C_k \rightarrow C_{k-1} \]
takes place at rate $K_{-R^{(k)}_2}=\alpha e^{E_k- E_{k-1} }$ for every $k \in \{0, \dots , N \} $.

In order to simplify the model we make the following assumptions
\begin{enumerate}
    \item $E_S =0$ and $E_0=\sigma $ and $E_k = \sigma + k E $ for every $ k \in \{ 1, \dots, N\} $ for $E>0$. As a consequence the rate of the reactions $(k) \rightarrow (k-1) $ is $\alpha e^{E} $. 
    \item The ligand is degraded at rate $\mu $ independently on its state. 
    Moreover we assume that the degradation rate $\mu $ is such that 
\begin{equation}\label{degradation}
\mu \sim e^{- b N } \text{ as } N \to \infty
\end{equation}
where $b >0$. 
\end{enumerate}
These assumptions guarantee that for every cycle
 $\omega _k=[R_1^{(k)}, R_2^{(k)}, -R_1^{(k+1)}]$ of the form 
\[
S \overset{R_1^{(k)}} \longrightarrow C_k  \overset{R_2^{(k)}}\longrightarrow C_{k+1}  \overset{- R_1^{(k+1)}}\longrightarrow S
\]
it holds that 
\begin{equation} \label{lack of db}
e^{\Delta} = \frac{ K_{R_1^{(k)}}K_{R_2^{(k)}} K_{-R_1^{(k+1)}}  }{ K_{- R_1^{(k)} }K_{- R_2^{(k)} }K_{ R_1^{(+1k)} }}. 
\end{equation}
Hence when $\Delta \neq 0$, the detailed balance property does not hold. 
We stress that we are assuming here that the only reactions that are affected by the parameter $\Delta $ are the phosphorylation reactions $(k ) \rightarrow (k+1) $. The fact that detailed balance fails along these reactions implies that the rate $\alpha e^\Delta$ of the phosphorylation reactions is larger than the rate that we have when detailed balance holds, i.e. $\alpha $. In other words, the lack of detailed balance accelerates the phosphorylation reactions.  

Let $n_k (t) $ be the probability  that a ligand, characterized by the binding energy $\sigma$, reaches the state $k \in \{ 0, \dots , N\} $ in the time interval $(0, t ] $ and let $n_S (t) $ be the probability that it reaches the state $S$ in the time interval $(0, t ] $. 
The dynamics of the vector of the probabilities $n(t)=(n_S(t), n_0(t), n_1(t), \dots , n_N(t) )^T \in \mathbb R_*^{N+1}$ is described by the following system of ODEs 
\begin{align}\label{ODEs}
    \frac{d n_S }{dt} &= - n_S  - S_N(E) n_S + e^{\sigma }  \sum_{k=0}^{N} n_k - \mu n_S \nonumber \\
    \frac{d n_0 }{dt} &= n_S  - e^{\sigma } n_0 + \alpha e^E n_1 - \alpha e^{\Delta} n_0 - \mu n_0 \nonumber \\
        \frac{d n_k }{dt} &= n_S e^{- k E} - \left( e^{\sigma } + \alpha e^E + \alpha e^\Delta  \right)  n_k + \alpha e^E n_{k+1}  + \alpha e^{\Delta} n_{k-1} - \mu n_k  \\
         \frac{d n_{N} }{dt} &= n_S e^{- N E} - \left( e^{\sigma } + \alpha e^E + \alpha e^\Delta  \right)  n_{N} + \alpha e^{\Delta} n_{N-1} - \mu n_{N} \nonumber  
\end{align}
with initial datum $n(0)=(1, 0, \dots, 0)$ and 
where 
\[
S_N(E):=\sum_{k=1}^{N} e^{- k E } = \frac{1-e^{-NE}}{e^E-1 } . 
\]
The fact that we are considering as initial datum $n(0)=(1, 0, \dots, 0)$ means that at time $t=0$ the only substances present in the network are the free ligands. 
The probability that a ligand, characterised by the binding energy $\sigma$, produces a response in the time interval $(0,t]$, denoted with $R(t)$, is given by 
\[ 
R(t) := \alpha e^\Delta n_N(t). 
\]
The probability $p_{deg} (t)$ that a ligand is degraded in the time interval $(0, t] $ is 
\[
p_{deg} (t)  := - \mu M(t)  
\]
where $M(t)$ can be interpreted as the probability that the ligand is the network (either in a complex or as a free ligand) in the time interval $(0, t] $ 
\begin{equation} \label{def M}
M(t):=n_S(t) +\sum_{k=0}^{N} n_k (t) .
\end{equation}
Later we refer to $M $ as the total mass of ligands in the network. 
According to the system of equations introduced above we have that $M$ is decreasing in time. More precisely we have that 
\begin{equation} \label{mass}
\frac{d M }{ dt } = - \alpha e^\Delta n_{N} - \mu M . 
\end{equation}
Finally notice that, as expected, we have that
\[ 
p_{deg}(t) + R(t) + M(t)=1, \quad  \forall t \geq 0. 
\]
In this paper we analyse the probability  $p_{res}(\sigma)$ that one ligand characterized by its binding energy $\sigma $ with the receptor produce a responses, i.e.   
\begin{equation} \label{prob response}
p_{res} (\sigma ) := \int_0^\infty R(t) dt =  \alpha e^\Delta \int_0^\infty n_{N} (t) dt
\end{equation} as the number of kinetic proofreading steps $N \to \infty $.

\subsection{Strong discrimination properties for $\Delta > \Delta_c$ }
The goal of the paper is to analyse the relation between the lack of detailed balance and strong discrimination properties of kinetic proofreading mechanisms. We state here  precisely our definition of strong discrimination. 
\begin{definition}[Strong discrimination] \label{def:strong discr}
Let $\Delta, \sigma, b, E , \alpha  $ be positive constants. 
    The kinetic proofreading network described by the system of ODEs \eqref{ODEs} has strong discrimination properties with respect to the energy $\sigma$  if 
\begin{equation} \label{strong disc}
\lim_{N \to \infty}    \frac{1}{N } \log \left( \frac{1}{p_{res} (\sigma) } - 1 \right) = \lambda_\Delta (\sigma, E) - b 
\end{equation}
    where  $ \lambda_\Delta (\sigma, E )$ is a parameter that depends non trivially on  $\sigma $  and there exists a solution $\sigma_c $ to 
    \begin{equation} \label{sigma cirt}
\lambda_\Delta(\sigma_c, E)= b. 
\end{equation} 
\end{definition}
In this paper we prove that, for the kinetic proofreading networks introduced in Section \ref{sec:model}, we have that
\begin{equation} \label{lambda}
\lambda_\Delta  (\sigma, E )=\begin{cases}  E &\text{ if } \Delta < \Delta_c \\
\psi_\Delta (\sigma, E) &\text{ if } \Delta > \Delta_c
\end{cases} 
\end{equation}
where $\psi$ is a function which depends non trivially on $\sigma$, $E$ and $\Delta$
\begin{equation} \label{psi}
\psi_\Delta(\sigma , E ) := E - \log \left[ \frac{1}{2 \alpha } \left( e^\sigma + \alpha (e^E+ e^\Delta) - \sqrt{ (e^\sigma + \alpha (e^E+ e^\Delta) )^2 - 4 \alpha^2 e^{\Delta + E }} \right) \right]
\end{equation}
and where 
\begin{equation}\label{Delta_c}
\Delta_c= \Delta_c (\sigma)  := \log \left( 1+ \frac{e^\sigma }{\alpha (e^E-1 ) } \right).  
\end{equation}
In the latter, for notational convenience, we will remove the subscript $\Delta $ from $\lambda_\Delta $ and from $\psi_\Delta $.

The fact that $\lambda(\sigma, E) $  does not depend on $\sigma $ for $\Delta $ and $\sigma$ such that $\Delta < \Delta_c(\sigma) $ and depends on $\sigma$ for $\Delta $ and $\sigma $ such that $\Delta > \Delta_c(\sigma)  $ is the most important result of this paper. It implies that the kinetic proofreading networks described by \eqref{ODEs} can have strong discrimination properties, in the sense of Definition \ref{def:strong discr}, only when $\Delta $ and $\sigma$ are such that $\Delta > \Delta_c(\sigma) $. 
\begin{figure}[H] 
\centering
\includegraphics[width=0.5\linewidth]{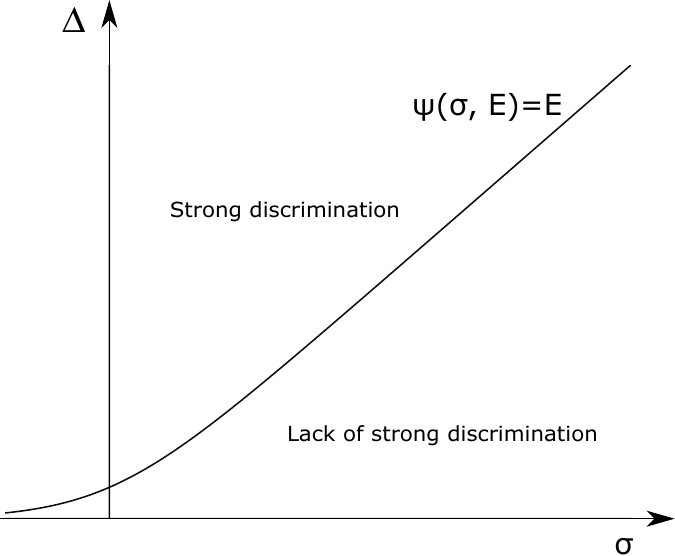}
\caption{
In this figure we plot the line $\psi (\sigma, E)=E $ for $\alpha =1 $, $E=\ln (2)$. Notice that $\psi(\sigma , E )=E $ if and only if $\Delta= \Delta_c(\sigma)$. The line $\psi (\sigma, E)=E $ separates the region in which we have strong discrimination properties from the region where we do not have strong discrimination. 
}\label{fig1}
\end{figure}

\subsection{Necessary conditions on $b$ for strong discrimination when $\Delta > \Delta_c $}
The goal of this section is to explain how to interpret the consequences of the definition of strong discrimination given in Section \ref{sec:discrimination} and in particular to introduce a condition on the parameter $b$ that is necessary in order to have that equation \eqref{sigma cirt} admits a solution, hence that the strong discrimination property holds. 

The fact that the parameter $\lambda(\sigma, E) $ depends on $\sigma $ is crucial in order to have a sharp transition of $p_{res} (\sigma) $ from $1$ to $0$ around a critical binding energy $\sigma_c$. 
Indeed \eqref{strong disc} implies that the probability $p_{res} (\sigma) $ that one ligand characterized by the binding energy $\sigma $ satisfies 
\begin{equation} \label{eq:prob response}
p_{res} (\sigma)   \sim \left( 1 + Ce^{N(\lambda(\sigma, E)- b)}\right)^{-1} \text{ as } N \to \infty
\end{equation}
where $C$ is a positive constant that depends on the parameters. When $\Delta< \Delta_c(\sigma) $ we have that $\lambda (\sigma , E )=E $.
As a consequence, if $\lambda(\sigma, E )=E < b $, then $\lim_{N\to \infty } p_{res} (\sigma ) =1 $, while when $\lambda(\sigma , E )=E > b $, then $\lim_{N\to \infty } p_{res} (\sigma ) =0 $. 
As will be explained in detail in Section \ref{sec:Delta to infty}, when $\Delta \to \infty $ and $\alpha, E , \sigma $ are of order one the function $\psi $ does not depend on $ \sigma $ and we have a similar behaviour. The model in this case does not have strong discrimination properties in the sense of Definition \ref{def:strong discr}. 

Consider a binding energy $ \sigma_c \in \mathbb R $. Assume that $\Delta > \Delta_c = \Delta_c( \sigma_c) $. By the continuity of the function  $\Delta_c(\sigma) $ defined by \eqref{Delta_c} there exists a neighborhood  $U$ of $\sigma_c $ such that $\lambda (\sigma, E )$ is a strictly increasing function of $\sigma$ in $U$.
Then we can find a parameter $b>0 $ such that \eqref{sigma cirt} holds. 
Precise conditions that the parameter $b$ must satisfy are given in  \eqref{conditions on b}. In particular we must have that $b< E $. The reason is that $\lambda(\sigma_c, E)=\psi(\sigma_c, E) $ is a decreasing function of $\Delta$. Therefore, in order to have that $(\sigma_c, \Delta) $ are in the region of strong discrimination in Figure \ref{fig1} we must have $\psi(\sigma_c, E ) = b < E $.

The condition $b < E$, needed to have strong discrimination has an interpretation: the characteristic life time of the ligands $e^{bN}$ must be strictly smaller than the characteristic time to reach the state $N$ directly from the signal $S$, i.e. $e^{EN} $ as $N \to \infty $.

The fact that there exists a $\sigma_c $ such that  \eqref{sigma cirt} holds implies that there exists a region of transition from response ($p_{res}(\sigma ) =1$) to non-response ($p_{res}(\sigma ) =0$) when $\sigma $ is such that 
\[ 
|\lambda (\sigma, E ) - \lambda(\sigma_c, E ) |= |\lambda (\sigma , E) -b | \lesssim \frac{1}{N}.
\] 
Moreover we have that if $\sigma > \sigma_c $ is such that $ | \lambda (\sigma , E) - \lambda(\sigma_c, E) | \gtrsim \frac{1}{N} $, then $\lim_{N\to \infty } p_{res} (\sigma ) =0 $ while when $\sigma < \sigma_c $ is such that $| \lambda (\sigma_c , E) - \lambda(\sigma, E ) | \gtrsim \frac{1}{N} $, then $\lim_{N\to \infty } p_{res} (\sigma ) =1 $. 

In Definition \ref{def:strong discr} we say that the kinetic proofreading model has strong discrimination properties if $\lambda(\sigma, E) $ depends on $\sigma$ and when there exists a solution to \eqref{sigma cirt}, because in this case there exists a critical binding energy $\sigma_c $ which separates the binding energies that produce response from the binding energies that do not produce a response.
Moreover the size of the region of transition from response to non-response, in the space of binding energies $\sigma$, where errors in the detection of ligands can occur, is of order $1/N$, hence tends to zero as $N \to \infty $. In Figure \ref{fig1} we compare the probability of response $p_{res} (\sigma ) $ that we obtain selecting the parameters in such a way that $\Delta > \Delta_c $ with the probability of response when $\Delta < \Delta_c$.
Notice that in the latter case the probability is almost constantly equal to $0$ while in the first case we have a sharp transition from response to non response around a critical binding energy $\sigma_c$. 

Finally let us consider the limiting case in which $\Delta \to \infty$ and $\alpha, \sigma $ and $E$ are of order one. Under these assumptions we have that 
\[
 \lim_{\Delta \to \infty }   \psi(\sigma , E )=  E.  
\]
Therefore in this regime $\psi (\sigma, E) $ does not depend on $\sigma$, hence we do not have strong discrimination properties. More details on the limiting cases with $\Delta \to \infty $ will be presented in Section \ref{sec:alterantive models no db and no disc}.  
\begin{figure}[H] 
\centering
\includegraphics[width=0.6\linewidth]{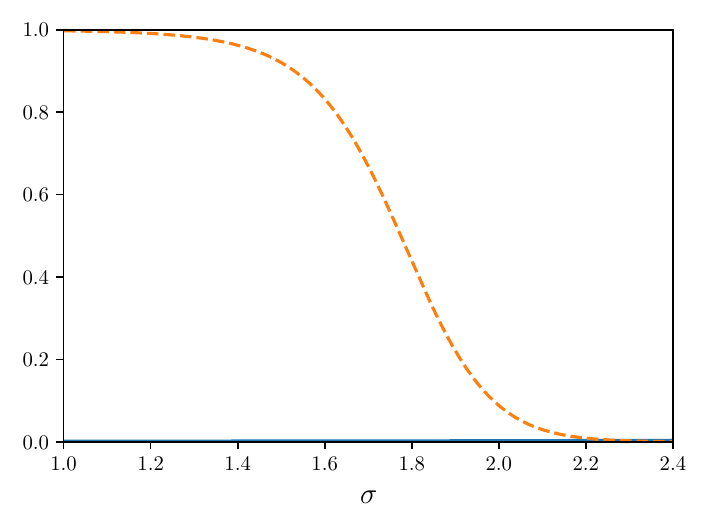}
\caption{The continuous line in blue is the probability of response $p_{resp} (\sigma) $ when $\alpha =1 $, $E=\log(3)$, $b= \log(2)$, $\Delta=1/10$ and $N=20$. In this regime we have that $\Delta$ and $\sigma $ are such that $\Delta < \Delta_c(\sigma) $, hence we are in the subcritical case and since $b < E $ we have $p_{res} \approx 0$. 
The dashed line in orange is a plot of  the probability of response when $\alpha, E, b , N $ are as before, but $\Delta=2$. 
In this regime $\Delta$ and $\sigma $ are such that $\Delta > \Delta_c(\sigma) $, hence we are in the supercritical case.
}\label{fig2}
\end{figure}

\section{Kinetic proofreading models out equilibrium} \label{sec:completion}
In this paper we prove that the class of linear kinetic proofreading models described by \eqref{ODEs} must have a sufficiently large amount of lack of detailed balance in order to have strong discrimination properties. 
However, at the fundamental level, every chemical reaction network at constant temperature and that does not exchange chemicals with the environment must satisfy the property of detailed balance (see \cite{franco2025detailed,polettini2014irreversible}). The kinetic proofreading models that we consider in this paper can be seen as  effective systems that approximate the dynamics of  larger chemical networks that exchange chemicals with the environment and where the detailed balance property holds. Hence the lack of detailed balance is justified by the fact that the system is assumed to be in contact with reservoirs of chemicals, in particular of molecules of $ATP $, $ADP $ and phosphate groups $P$. 

\subsection{Kinetic proofreading model involving molecules of ATP,  ADP and phosphate groups}
In this section we derive the kinetic proofreading model introduced in Section \ref{sec:model} starting from a chemical reaction network in which we include explicitly the molecules of $ATP$, $ADP $ and the phosphate groups. This enlarged chemical reaction network satisfies the detailed balance property and reduces to the model studied in this paper if we assume that the concentrations of  $ATP $, $ADP $ and  $P$ are kept at constant levels by an active process, for instance $ATP $ production by the mitochondria (that we do not model in detail in this paper). 
We will also  explain that the parameter $\Delta $, measuring the lack of detailed balance, depends on the values of the frozen concentrations of $ATP $, $ADP $ and $P$. 

Since we are interested in the property of detailed balance it is convenient to write the linear kinetic proofreading model described in Subsection \ref{sec:model} as a chain of cycles $ \{ \omega_k \}_{k=0}^{N} $. These cycles have the following form 
\begin{equation}\label{cycle to complete}
S \overset{e^{- k E} }{\underset{e^\sigma} \rightleftarrows} C_k  \overset{\alpha e^\Delta }{\underset{\alpha e^E }\rightleftarrows } C_{k+1} \overset{e^\sigma}{\underset{e^{- (k+1)E}} \rightleftarrows} S. 
\end{equation}
Notice that if $\Delta =0$, then we have that the Wegscheider condition \eqref{weg condition} holds for every cycle, hence the property of detailed balance holds for the chemical reaction network. 
 
In order to obtain the cycles $\{ \omega_k \}_{k=0}^N$ in which the Wegscheider criterium fails, starting from cycles for which the Wegscheider criterium holds, we consider the chemical reaction network with substances $\Omega_c = \Omega \cup \{  ATP , ADP, P, S  \} $. We assume that these substances interact via the chemical reactions $ \{ R_1^{(k)}, -R_{1}^{(k)} \}_{k=0}^{N}$, $R_2, - R_{2} $, $ \{ R_3^{(k)}, - R_{3}^{(k)} \}_{k=0}^{N}$ defined as 
\begin{equation} \label{completed} 
(k) + (ATP) \overset{R_1^{(k)}}{\underset{- R_{1}^{(k)}} \rightleftarrows}(k+1)+ (ADP), \quad (ATP) \overset{R_2}{\underset{-R_{2}} \rightleftarrows} (ADP) +(P), \quad  (k) \overset{R_3^{(k)}}{\underset{-R_{3}^{(k)}} \rightleftarrows}  (S) + k * (P), 
\end{equation}
for $  k \in  \{ 0, \dots, N\}.$ 
Here we are using the notation
\[ 
(k) *(P) = \overbrace{(P) + (P) + \dots + (P)}^{ k  \text{ times} }
\] 
We assume that the chemical rates of the reactions $ \{ R_1^{(k)},-  R_{1}^{(k)} \}_{k=0}^{N}$ are independent on $k =0, \dots N$.  
More precisely, the chemical rates of the reactions $R_1^{(k)}, R_2$ and of the reverse reactions $-R_{1}^{(k)} $ and $ -R_{2}$ are given by 
\[
K_{R_1^{(k)}}=K_1 =\alpha e^{E_T}  , \ K_{-R_1^{(k)}}=K_{-1 }= \alpha e^{ E+ E_D }, \quad K_{R_2}=e^{E_T},\  K_{-R_2} = e^{ E_D+E_P },
\] 
for some constants $E_T, E_D, E_P \in \mathbb R$. 
The chemical rates of the reactions $R_3^{(k)} $ and of the reverse reactions $R_{-3}^{(k)}$ are given by 
\[
K_{R_3^{(k)}}= K_3=e^\sigma ,\  K_{-R_3^{(k)}}= K_{-3}= e^{ k (E_P- E)},  \quad k \in \{ 0, \dots, N\}. 
\]
\paragraph{Cycles of the enlarged chemical reaction network.}
Notice that if we apply the reactions $R^{(k)}_1, R_{-2}$, $ R_{-3}^{(k)}, R_3^{(k+1)} $ to any vector of concentrations, then we will have a null effect, indeed
\[
(0,0,0,0,0, 0) \overset{R_1^{(k)}}  \rightarrow  (-1,1,-1,1,0, 0)  \overset{ R_{-2}} \rightarrow (-1,1,0,0,-1, 0)  \overset{R_{-3}^{(k)}}\rightarrow  (0,1,0,0,-(k+1), -1 )  \overset{R_{3}^{(k+1)}} \rightarrow (0,0,0,0,0, 0). 
\]
Therefore the set of the reactions $ \overline \omega_k = \{ R_1^{(k)}, R_{-2},  R_{-3}^{(k)}, R_3^{(k+1)}  \} $ for $k =0, 1, \dots, N$  are cycles of the enlarged chemical reaction network. 
Notice that by construction we have that the circuit condition \eqref{weg condition} holds along every cycle $\overline \omega_k $, indeed the rates were selected in such a way that
\[
K_1 K_{-2 } K_{- 3 }^{(k)} K_3 = K_{-1} K_{2 } K_{ 3 } K_{-3}^{(k+1)}, \quad  \forall k \in \{0, \dots, N\} .  
\]
As a consequence, this enlarged chemical reaction network satisfies the property of detailed balance. 

\paragraph{System of ODEs associated with the enlarged chemical reaction network and its conserved quantities.}
The system of ODEs associated with the chemical network that includes ATP, ADP and P is the following, 
\begin{align} \label{ode completed}
    \frac{d n_0 }{dt } &= - n_0 ( \alpha e^{E_T}  n_T + e^\sigma  ) + \alpha e^{ E+ E_D } n_{1} n_D + n_S \nonumber  \\
      \frac{d n_{k} }{dt } &= - n_{k} ( \alpha e^{E_T} n_T +  e^\sigma) + \alpha e^{ E+ E_D }n_{k+1} n_T + e^{(k+1)(E_p- E)} n_p^{k+1} n_S, \quad k \in \{ 1, \dots, N \} \nonumber \\
          \frac{d n_T }{dt } &= - n_T \left(e^{E_T} + \alpha e^{E_T} \sum_{k=0}^{N-1} n_k \right) + \alpha e^{ E+ E_D } n_D \sum_{k=1}^{N}n_{k}  + e^{ E_D+E_P } n_P n_D  \\ 
            \frac{d n_D }{dt } &= - n_D \left( e^{ E_D+E_P } n_P + \alpha e^{ E+ E_D } \sum_{k=1}^{N} n_{k} \right) + \alpha e^{E_T}  n_T \sum_{k=0}^{N-1}n_{k}  + e^{E_T} n_T  \nonumber \\ 
               \frac{d n_P }{dt } &= - e^{ E_D+E_P } n_P n_D - n_S  \sum_{k=0}^{N}  k e^{k (E_p - E)} n_P^k + e^{E_T} n_T + e^\sigma \sum_{k=0}^{N} k  n_k \nonumber \\ 
      \frac{d n_S}{dt }&= e^\sigma \sum_{k=0}^{N} n_k  - n_S  \sum_{k=0}^{N} e^{k (E_p - E)} n_P^k . \nonumber
      \end{align}

Notice that the vector $\textbf{E} \in \mathbb R^{N+4} $ defined as 
\[
\textbf{E} =(\sigma, \sigma + E, \sigma +2E, \dots ,  , \sigma + (N-1) E, \sigma + N  E, E_T , E_D, E_P , 1 )
\]
is such that $e^{ - \textbf{E}} $ a steady state for the system of equations \eqref{ode completed}. 

Moreover notice that we have some conserved quantities in our system. 
In particular we have that if $n(t)=(n_0, n_1, \dots, n_{N-1} , n_N, n_T, n_D, n_P, n_S ) \in \mathbb R^{N+ 4 } $ is the vector of the concentrations, then we have that 
\[
m_1^T n (t) = m_1^T n(0), \quad m_2^T n (t) = m_2^T n(0)
\]
where $m_1=(\overbrace{1,1,\dots, 1 ,1}^{ N \text{ times}},0,0,0,1)^T $ and $m_2 = (0,1,2, 3, \dots,N-1, N, 2,1 , 1,0 )^T $. 
In other words we have that the number of phosphate groups is constant in time, hence 
\[ 
n_T(t) +2 n_{D} (t)+ n_P (t)+ \sum_{k=0}^N k n_k(t)   = n_T (0)+2 n_{D} (0) + n_P (0) + \sum_{k=0}^N k n_k (0)
\] 
 and the number of ligands is constant in time, i.e. 
\[
\sum_{k=0}^N  n_{k} (t) + n_S  (t) = \sum_{k=0}^N  n_{k} (0) + n_S  (0). 
\] 
Notice that all the conserved quantities, i.e. all the vectors $m \in \mathbb R^{ N + 4} $ such that $m^T n (t) = m^T n(0)$, can be written as linear combinations of $m_1$ and of $m_2$.

Since the chemical reaction network \eqref{ode completed} satisfies the property of detailed balance all the steady states of \eqref{ode completed} are of the form 
\[
N = e^{- \textbf{E} + \mu_1 m_1  + \mu_2 m_2  }
\]
for suitable constants $\mu_1, \mu_2 $, see Lemma 3.10 in \cite{franco2025detailed}. 

Therefore the concentrations of $ATP$, $ADP $ and $P$ at any steady state $N \in \mathbb R_+^{N +4 } $ of \eqref{ode completed}, denoted respectively by $N_{T}, N_D, N_p $, satisfy the following equality 
\[
\frac{N_T }{N_P N_D } = e^{  E_D + E_P -E_T  }. 
\]

\subsection{Lack of detailed balance for constant non equilibrium concentration of ATP}
We now assume that the concentrations of $ATP$, $ADP $ and $P$ are constant in time. This can be justified when the chemical network \eqref{ode completed} is in contact with a reservoirs of $ATP$, $ADP $ and $P$.
Let $\overline n_T $, $\overline n_D $ and $ \overline n_P$ be the constant concentrations of $ATP$, $ADP $ and $P$ respectively. 
Let us define $\Delta $ as
 \[
\Delta = E_T - E_D - E_P + \log \left( \frac{\overline n_T }{ \overline n_P \overline n_D }  \right). 
 \] 
Since the concentrations of $ATP$, $ADP $ and $P$ are constant in time, the chemical reactions in \eqref{completed}  can be reduced to 
\begin{equation} \label{reduced chem net}
(k) \overset{\overline R_1^{(k)} }{\underset{\overline R_{-1}^{(k)} } \rightleftarrows}(k+1),  \quad  (k) \overset{ \overline R_3^{(k)} }{\underset{\overline R_{-3}^{(k)} } \rightleftarrows}  (S)  \quad k \in \{ 0, \dots, N\}
\end{equation}
where  we assume that the chemical rates $\overline{K}_1, \overline{K}_{-1}$, $\overline{K}_3 , \overline{K}_{-3}^{(k)}$ of the reactions depend on $\overline n_T,\overline n_D,\overline n_P$, more precisely 
\[
\overline{K}_1 = \alpha  e^{E_T}\overline n_T , \quad \overline K_{-1}= \alpha e^{E+ E_D }\overline n_D, \quad 
\overline K_3=  e^\sigma , \quad \overline K_{-3}^{(k)} = e^{k (E_P- E) } \overline n_P^k. 
\]
Notice that if we choose that the concentrations of $ATP$, $ADP $ and $P$ are given by 
\[ 
\overline n_P = e^{- E_P} , \quad  \overline n_D=e^{- E_D}, \quad \overline n_T = e^\Delta e^{- E_T} 
\] 
then the chemical reaction network  \eqref{reduced chem net} satisfies the property of detailed balance if and only if $\Delta = 0$. Moreover notice that, under these assumptions on the frozen concentrations, we have that  the chemical network  \eqref{reduced chem net} and the chemical reaction network \eqref{cycle to complete} coincide. 

Notice that the concentration of molecules of ATP, ADP and  phosphate groups are assumed to be constant in time because the system is in contact with reservoirs of these substances. 
Therefore we must have fluxes of ATP, ADP and P between the chemical network and the environment. 
We can rewrite the equations for the change in time of ATP, ADP and P in equation \eqref{ode completed} taking into account of these fluxes. 
The equations that we obtain are the following 
\begin{align} \label{ode completed fluxes}
    \frac{d n_T }{dt } &= - n_T \left(e^{E_T} + \alpha e^{E_T} \sum_{k=0}^{N-1} n_k \right) + \alpha e^{ E+ E_D } n_D \sum_{k=1}^{N}n_{k}  + e^{ E_D+E_P } n_P n_D +  J^{ext}_T(n)  \nonumber   \\ 
            \frac{d n_D }{dt } &= - n_D \left( e^{ E_D+E_P } n_P + \alpha e^{ E+ E_D } \sum_{k=1}^{N} n_{k} \right) + \alpha e^{E_T}  n_T \sum_{k=0}^{N-1}n_{k}  + e^{E_T} n_T  +   J^{ext}_D(n) \\
               \frac{d n_P }{dt } &= - e^{ E_D+E_P } n_P n_D - n_S  \sum_{k=0}^{N}  k e^{k (E_p - E)} n_P^k + e^{E_T} n_T + e^\sigma \sum_{k=0}^{N} k  n_k +    J^{ext}_P(n). \nonumber
\end{align}
In the equations above we have used the fact that $\overline n_P = e^{- E_P} $, $\overline n_D = e^{- E_D} $ and $\overline n_T = e^{- E_T + \Delta } $. 
Evaluating the fluxes at the steady states $N$ with  $N_T = e^{- E_T}$, $ N_D = e^{- E_D}$, $ N_P = e^{- E_P}$ that the fluxes must be given by 
\begin{align*} 
 J^{ext}_T(N)  &=  e^{\Delta } -1 + \alpha e^{- \sigma} \left( \sum_{k=
 0}^{N-1}e^{- k E} - e^\Delta  \sum_{k=
 1}^{N}e^{- k E} \right)   = \left( e^{\Delta } -1 \right) \left( 1 + \alpha e^{- \sigma} \frac{1-e^{- NE} }{ 1-e^{-E}} \right) > 0  \\
  J^{ext}_D(N)  &=   1  - e^{\Delta}  + \alpha e^{- \sigma}  e^E \sum_{k=1}^{N} e^{- k E} - e^\Delta \sum_{k=0}^{N-1} e^{- k E}  =  (1  - e^{\Delta} )\left( 1 + \alpha e^{- \sigma} \frac{1-e^{- NE} }{ 1-e^{-E}} \right)  <0 \\
   J^{ext}_P(N)  &=  1 - e^\Delta +  n_S \sum_{k=0}^{N} k e^{- k E}  -\sum_{k=0}^{N} k e^{- k E} = 1- e^\Delta <0. 
\end{align*}
The fluxes $J^{ext}_T(N) $, $J^{ext}_D(N) $, $J^{ext}_P(N) $ can be used in order to compute the energy needed in order to main the system at steady state.

\section{Matched asymptotic expansions} \label{sec:formal}
In this section we explain, using formal arguments, why the kinetic proofreading model has strong discrimination properties if and only if $\Delta > \Delta_c $.

In order to do this we study the probability of response $p_{res} (\sigma) $ defined as in \eqref{prob response} for large values of $N$. 
Therefore we analyse the behaviour of
$n_{N} $ as $N \to \infty $.
To this end we use the method of matched asymptotic expansion. More precisely we study the behaviour of $n_k$ as $N \to \infty $ in two main regions: the outer region, where $k$ is such that $N- k=m$ and $m \gg 1 $ as $N \to \infty $ and the inner region, where $k$ is such that $N- k \approx 1 $ as $N \to \infty $.
To obtain the asymptotic behaviour of $n_{N} $ as $N \to \infty $ we match the behaviours that we obtain in the inner and in the outer region. 

Notice that in this paper we study the discrimination properties of the linear kinetic proofreading model when the number of kinetic proofreading steps $N$ tends to infinity. 
The number of kinetic proofreading steps has been measured experimentally by many authors and has been estimated to be between 2 and 20 (see for instance \cite{altan2005modeling,britain2022progressive,pettmann2021discriminatory,tischer2019light}). 
Therefore the discrimination property in concrete biological situations will not be as strong as in the model analysed here. However, considering the limit as $N \to \infty $ allows to understand in detail how the kinetic proofreading mechanism performs its functions. In particular allows us to prove that a critical lack of detailed balance is needed in order to have strong discrimination properties.

Before analysing the behaviour of $n_k $ in the inner and in the outer region we study how to approximate the concentration of ligands, $n_S$ as $N \to \infty $. To this end we consider the limit as $N \to \infty$ in the equation for $n_S $ in \eqref{ODEs} and using the fact that $\mu $ satisfies \eqref{degradation} we obtain  
\[
\frac{d n_S(t) }{dt} = - \left( \frac{1}{1- e^{- E }} + e^\sigma \right) n_S(t) + e^{\sigma }  M(t)
\] 
where we recall that $M(t) $ is defined as \eqref{def M}. 

As we will show later, the concentrations $n_k $ stabilize to exponential solutions that behave as $e^{- \lambda(\sigma, E) k } $ for $k \approx N $ and $N $ large, at time scales that are of order $N $. Here $\lambda(\sigma, E) $ is positive and depends also on $\Delta$. 
We now argue heuristically to infer that the probabilities $n_k$ can be approximated by quasi stationary solutions $\tilde{n}_k$ solving 
\[
\tilde{n}_S e^{- k E} - \left( e^{\sigma } + \alpha e^E + \alpha e^\Delta  \right)  \tilde{n}_k + \alpha e^E \tilde{n}_{k+1}  + \alpha e^{\Delta} \tilde{n}_{k-1} =0,  \quad k \geq 1 
\] and such that their total mass decreases in time scales $t $ larger than $N$. 
Indeed, assume that $k=Nx $ and that $g(t, x)=\frac{ n_{k }(t) }{\tilde{n}_k}$. Then, in order to obtain the time scale for the stabilization of $n_k $ we approximate the discrete equation in \eqref{ODEs} with the following PDE
\begin{align*}
\partial_t g(t, x)  & \approx\frac{n_{S}(t)}{\tilde{n}_{Nx}} e^{-NxE}  -  \alpha (e^E- e^\Delta) \frac{1}{N} \partial_x g(t,x)  \text{ as } N \to \infty.
\end{align*}

Here we are using the rough approximation $ \frac{\tilde{n}_{k} }{n_k } \approx 1 $ in order to approximate the incremental quotient by derivatives. Notice that the time scales at which the characteristic curves of this equation stabilize for $t \approx N $. Moreover it stabilizes to  $\frac{n_{S}(t)}{\tilde{n}_{Nx}} e^{-NxE} $. 
In Section \ref{sec:PDEs} we explain in detail under which conditions the discrete system of ODEs \eqref{ODEs} can be approximated by means of PDEs. 
A more detailed analysis of the time dependent discrete  equation \eqref{ODEs} gives the same order of magnitude for the stabilization time. In Section \ref{sec:rigorous} we obtain rigorously that the stabilization of  $n_k e^{ -k \lambda (\sigma, E) }$ takes place for times of order $N$.

In particular the equation for $M$, i.e. \eqref{mass}, implies that the variation of mass occurs at time scales of order $e^{ c  N } $ for a suitable positive constant $c>0 $ that depends on $E$, $\Delta$ and on $\sigma $. Notice also that $n_S$ stabilizes in times $t$ of order $1$. Therefore in order to obtain the variation of $ (n_S, n_k)_{k=0}^N $ in times scales of order $N $ we can use a quasi steady state approximation to replace $(n_S, n_k )_{k=0}^N $ by a steady state having mass $M (t)$ and to use the equation for $n_N $ in \eqref{ODEs} in order to compute the change of the mass in exponentially long time scales.   
Later we will show that these assumptions are consistent with the results that we obtain via this heuristic argument, see \eqref{consistency}.

In particular, the quasi steady state approximation for $n_S $ implies that $n_S \approx \tilde{n}_S $ where 
\[
\tilde{n}_S(t) \sim  \overline n_S  M(t) \text{ as } N \to \infty
\]
where 
\begin{equation} \label{overline nS}
 \overline n_S = \frac{e^\sigma }{ e^\sigma + \frac{e^E}{e^E-1}}.
\end{equation}

\subsection{Outer region: $N - k  \gg 1 $ as $N \to \infty$.}
In this section, we study the behaviour of $n_k$ when $N \to \infty$ for $k \approx 1$ such that $ k = N-m$  with $m \to \infty $ as $N \to \infty$. 
Since in this region we have that $k \ll N $, we assume that the dynamics  is given by the system of equations obtained taking the limit as $N \to \infty $ in \eqref{ODEs}, i.e. 
\begin{align} \label{ODE half line}
    \frac{d n_0^{(out)} }{dt} &= n_S  - e^{\sigma } n_0^{(out)} + \alpha e^E n_1^{(out)} - \alpha e^{\Delta} n_0^{(out)} \\
        \frac{d n_k^{(out)} }{dt} &= n_S e^{- k E} - \left( e^{\sigma } + \alpha e^E + \alpha e^\Delta  \right)  n_k^{(out)} + \alpha e^E n_{k+1}^{(out)}  + \alpha e^{\Delta} n_{k-1}^{(out)}, \quad k \in \mathbb N_0.   \nonumber
\end{align}
The initial datum for equation \eqref{ODE half line} is the same initial datum that we have for the original system of ODEs \eqref{ODEs}, i.e. we have $n^{(out)}(0)=(1, 0, \dots, 0)$. Finally notice also that 
\[
\frac{d}{dt} M^{(out)}(t)=\frac{d}{dt} n_S(t) +  \frac{d}{dt} \sum_{k=0}^\infty n_k^{(out)}(t) = 0. 
\]

 As before, we assume that $n_S$ and $ \{ n^{(out)}_k \}_{ k=0}^\infty $ can be approximated by quasi stationary solutions $\tilde{n}_S$ and $ \{ \tilde{n}^{(out)}_k \}_{ k=0}^\infty $. 
Later we will see that this assumption is consistent with our results. 
Therefore we look for solutions to the following system of equations
\begin{align} \label{ODE half line stationary}
   0 &= \tilde{n}_S - e^{\sigma } \tilde{n}_0^{(out)} + \alpha e^E \tilde{n}_1^{(out)} - \alpha e^{\Delta} \tilde{n}_0^{(out)}  \\
        0 &= \tilde{n}_S e^{- k E} - \left( e^{\sigma } + \alpha e^E + \alpha e^\Delta  \right)  \tilde{n}_k^{(out)} + \alpha e^E n_{k+1}^{(out)}  + \alpha e^{\Delta} \tilde{n}_{k-1}^{(out)}, \quad k \in \mathbb N_0. \label{ODE half line stationary2}
\end{align}
In order to study the solutions to this equation we consider separately the case of $\Delta = \Delta_c $ and $\Delta \neq \Delta_c$. 

As a first step we compute $\tilde{n}_k^{(out)}$ for $\Delta \neq \Delta_c$.
Then we will see that, as $N \to \infty$, we have different behaviours for $\tilde{n}_k^{(out)}$ depending on whether $\Delta> \Delta_c$ or $\Delta < \Delta_c $. 
For completeness we also study the case $\Delta = \Delta_c$.
\paragraph{Explicit solution to \eqref{ODE half line stationary} in the subcritical and supercritical case $\Delta \neq \Delta_c $.}
Notice that the solution to equation \eqref{ODE half line stationary} can be written as a linear combination of a particular solution and of the homogeneous solutions. Here with homogeneous solutions we mean solutions to equation \eqref{ODE half line stationary}-\eqref{ODE half line stationary2} where $\tilde{n}_S=0$, i.e. 
\begin{align} \label{ODE half line stationary homogeneous}
   0 &= - e^{\sigma } \tilde{n}_0^{(out)} + \alpha e^E \tilde{n}_1^{(out)} - \alpha e^{\Delta} \tilde{n}_0^{(out)} \\
        0 &= - \left( e^{\sigma } + \alpha e^E + \alpha e^\Delta  \right)  \tilde{n}_k^{(out)} + \alpha e^E \tilde{n}_{k+1}^{(out)}  + \alpha e^{\Delta} \tilde{n}_{k-1}^{(out)}, \quad k \in \mathbb N_0. \nonumber
\end{align}
It is easy to see that a particular solution to \eqref{ODE half line stationary2} is 
\[ 
n_k^{part} = e^{- k E } A(0)  \overline n_S M(t) = e^{- k E } A(0) \tilde{n}_S(t) , \quad  k \geq 1 
\] 
where the constant $A(0)$ is given by 
\[
A(0):= \frac{1}{e^\sigma - \alpha (e^\Delta-1) ( e^E-1 )}. 
\]
Notice that since we have that $\Delta \neq \Delta_c $, then we have that $e^\sigma - \alpha (e^\Delta-1) ( e^E-1 ) \neq 0.$

We now look for solutions to \eqref{ODE half line stationary homogeneous} of the form $n_k^{hom} = \theta(0)^k$. Substituting $n_k^{hom} $ in equation \eqref{ODE half line stationary homogeneous} we deduce that $\theta(0) $ must satisfy the following equation 
\begin{equation}\label{second order eq}
\theta^2 (0)- \frac{\Omega(0)}{\alpha e^E} \theta (0) + e^{\Delta - E} =0, 
\end{equation}
where 
\[
\Omega (0):= e^\sigma + \alpha (e^E + e^\Delta).
\] 
As a consequence we have that the solutions $\theta_{1}(0)$ and $\theta_2(0)$ to the equation \eqref{second order eq} are
\[
\theta_{12}(0)= \frac{1}{2 \alpha e^E } \left(  \Omega(0) \pm  \sqrt{\Omega(0)^2 - 4 \alpha^2  e^{\Delta + E} } \right). 
\]
Now notice that 
\[
\theta_1 (0)= \frac{1}{2 \alpha e^E } \left(  \Omega(0) -  \sqrt{\Omega(0)^2 - 4 \alpha^2  e^{\Delta + E} } \right) < 1 \text{ while } \theta_2 (0)= \frac{1}{2 \alpha e^E } \left(  \Omega(0) +  \sqrt{\Omega(0)^2 - 4 \alpha^2  e^{\Delta + E} } \right) > 1. 
\]
Indeed we have that the function $f(\theta)=\theta^2- \frac{\Omega(0)}{\alpha e^E} \theta  + e^{\Delta - E}$ is such that $f(1)=- e^\sigma< 0$. 
Since we have that $M^{(out)} (t) $ is finite, then the only admissible solution is $\theta_1(0)$, indeed $\theta_2^k(0) \rightarrow \infty $ as $k \to \infty $. 

We deduce that
\[
 \tilde{n}_k^{(out)}(t) = e^{- k E } A(0)  \overline n_S M(t)  + F(t) \theta_1^k(0),\  k \in \mathbb N. 
\]
where $F$ is a suitable function that varies slowly in time, as a consequence of the fact that, as explained at the beginning of Section \ref{sec:formal}, the time scale at which the total mass of chemicals in the system starts to decrease is exponential in $N$. 
In order to obtain the function $F$ we notice that the formula above implies that
\[
 \tilde{n}_0^{(out)}(t) =  A(0)  \tilde{n}_S(t)    + F(t),  \quad t>0. 
\]
On the other hand, equation \eqref{ODE half line stationary} together with the fact $  \tilde{n}_1^{(out)}(t) =e^{- E } A(0)    \tilde{n}_S(t)  + F(t) \theta_1(0)$  implies that
\begin{align*}
 \tilde{n}_0^{(out) } (t) &= \frac{ A (0)  \tilde{n}_S(t)+ \alpha e^E  \tilde{n}_1^{(out)}(t) }{e^\sigma + \alpha e^\Delta } = \frac{ A(0)  \tilde{n}_S(t) + \alpha e^E  \tilde{n}_1^{(out)} (t)}{e^\sigma + \alpha e^\Delta } \\
&= \frac{ \tilde{n}_S(t) A(0) + \alpha A(0)  \tilde{n}_S(t) + \alpha e^E F(t) \theta_1(0)  }{e^\sigma + \alpha e^\Delta }. 
\end{align*}
We deduce that
\[
F(t)= \frac{A(0)   \tilde{n}_S(t) (1+\alpha - e^\sigma - \alpha e^\Delta )  }{ e^\sigma + \alpha e^\Delta - \alpha \varphi (\sigma )}
\]
where  we have that 
\begin{equation} \label{varphi}
\varphi(\sigma): = \frac{1}{2 \alpha } \left( \Omega(0)  - \sqrt{ \Omega(0)^2- 4 \alpha^2 e^{\Delta + E }} \right). 
\end{equation}

Summarizing, we obtain the following approximation for $n_k^{(out)} $ valid when $t \approx N $
\begin{equation} \label{asymptotic}
n_k^{(out)} (t)  \sim \tilde{n}_k^{(out)} (t) = e^{- k E } A(0) \tilde{n}_S(t)  \left[ 1 + B \varphi(\sigma)^k \right],\  N- k \gg 1 , \quad  N \to \infty 
\end{equation}
for 
\begin{equation} \label{B} B= \frac{ 1+\alpha - e^\sigma - \alpha e^\Delta  }{ e^\sigma + \alpha e^\Delta - \alpha \varphi(\sigma) }.
\end{equation}
Notice that the quasi steady state approximation $n_k^{(out)} (t)  \approx \tilde{n}_k^{(out)} (t)$ and $n_S(t) \approx \tilde{n}_S (t) $ is valid for $t \approx N $.

\subsubsection{Asymptotic behaviour of $n_k^{(out)}$ when $\Delta < \Delta_c$ and $k \gg 1 $, $N-k \gg 1 $.}
We now want to study the asymptotic behaviour of $n_k^{(out)} $ as $k  \to \infty $. 
To this end we first of all assume that $\Delta < \Delta_c$. 
A direct computation shows that in this case we have that $\varphi(\sigma) < 1$. As a consequence we deduce that, in the subcritical case (i.e. when $\Delta < \Delta_c$) we have that
 \begin{equation} \label{k large N-k large delta small}
\tilde{n}_k^{(out)} (t) \sim A(0) \overline n_S  M(t)  e^{- k E } \text{ as } k \to \infty. 
\end{equation}
Using the quasi steady state approximation we deduce that 
$
n_k^{(out)} (t) \sim  A(0) \overline n_S  M(t)  e^{- k E } $  as $ k \gg 1 $, $N- k \gg 1$ and $ t \approx N $. 

\subsubsection{Asymptotic behaviour of $n_k^{(out)}$ when $\Delta > \Delta_c$ and $k \gg 1 $, $N-k \gg 1 $.}
 On the contrary, in the supercritical case in which $\Delta > \Delta_c $, we have that $\varphi(\sigma) > 1$, hence 
 \begin{equation} \label{k large N-k large delta large}
  \tilde{n}_k^{(out)} (t)  \sim A(0)  \overline n_S  M(t)  B e^{ - k \psi(\sigma, E) }  \text{ as } k \to \infty,
 \end{equation}
 where we recall that $B$ is defined by \eqref{B} and $\psi $ is given by \eqref{psi} and it depends non trivially on $\sigma$. The quasi steady state approximation implies that  $ n_k^{(out)} (t)  \sim A(0)  \overline n_S  M(t)  B e^{ - k \psi(\sigma, E) }  $ as $k \gg 1 $, $N-k \gg 1 $ and $t \approx N $. 
 
Notice that both in the subcritical and in the supercritical case we have that the asymptotic behaviour of $n_k^{(out)} $ as $k \to \infty $ is given by $F(t) e^{- \lambda (\sigma , E) k } $ for a suitable function $F(t)$ encoding the change of mass in the system and a suitable exponent $\lambda(\sigma, E )$. 
The main difference between the supercritical and the subcritical case is that in the first case, when $\Delta > \Delta_c$, we have that $\lambda(\sigma, E ) =  E- \log(\varphi(\sigma ) )  $, hence $\lambda (\sigma, E) $ depends on the binding energy $\sigma $. Instead in the case $\Delta < \Delta_c$ we have that $\lambda(\sigma, E) $ does not depend on the binding energy $\sigma$. Indeed in that case we have that $\lambda(\sigma, E) =  E$. As we will see later in Subsection \ref{sec:discrimination} the fact that $\lambda(\sigma, E)$  depends on $\sigma $ when $\Delta > \Delta_c$ is crucial in order to have strong discrimination properties.

\subsubsection{Asymptotic behaviour in the critical case $\Delta = \Delta_c$ when $k \gg 1 $ and $N-k \gg 1 $}
By completeness we also study the behaviour of $n_k $ as $k \to \infty $ in the critical case $\Delta = \Delta_c$. 
Under this assumption, we have that $\alpha (e^\Delta-1) (e^E-1) = e^\sigma$. 
As in the case $\Delta \neq \Delta_c $ we have that the behaviour in the outer region is described by equation \eqref{ODE half line stationary}. 
We will now compute an explicit solution. 

\paragraph{Explicit solution to \eqref{ODE half line stationary}.}
First of all we notice that when $\Delta = \Delta_c$ we have that 
\[
\tilde{n}_k^{part} (t) = \frac {  k e^{- k E } \tilde{n}_S(t)}{ \alpha (e^{\Delta + E}  -1 )}, \quad k \in \mathbb N
\] 
is a particular solution to  \eqref{ODE half line stationary2}. 
We now aim at finding a solution to the homogeneous equation corresponding to \eqref{ODE half line stationary2}, i.e. to the following equation
\[ 
 0 =- \alpha\left(1+e^{E+ \Delta}  \right)  \tilde{n}_k^{(out)} + \alpha e^E \tilde{n}_{k+1}^{(out)}  + \alpha e^{\Delta} \tilde{n}_{k-1}^{(out)}, \quad k \in \mathbb N_0. 
 \] 
In this case we have homogeneous solutions of the form $n_k^{(hom)} =\theta(0)^k $ where $\theta(0)$ is a solution to 
\[ 
\theta(0)^2 - (e^\Delta + e^{- E} ) \theta(0) + e^{\Delta - E } =0. 
\]
Notice that the two solutions $\theta_{1}(0)$ and $\theta_2(0)$ to the second order equation are given by 
 \[
\theta_2(0)=e^\Delta, \quad  \theta_1(0)=e^{- E}. 
 \]
 As before, the only admissible solution is $\theta_1(0)< 1 $. 
 Therefore, the solution to equation \eqref{ODE half line stationary2} is
 \begin{equation} \label{eq:n_k inner delta_c}
 \tilde{n}_k^{(out)}(t) =e^{-k E } \left(  \frac{k \tilde{n}_S(t)  }{ \alpha (e^{\Delta + E}  -1 )}+ F(t) \right), \quad k \in \mathbb N 
 \end{equation}
 where $F(t)=\tilde{n}_0^{(out)}(t)$ and where $\tilde{n}_0^{(out)} (t)$ can be computed by solving equation \eqref{ODE half line stationary}. Indeed we have that 
 \[
 0=\tilde{n}_S (t) -(e^\sigma +\alpha e^\Delta) \tilde{n}_0^{(out)}  (t) + \alpha e^E \tilde{n}_1^{(out)} (t) = \tilde{n}_S (t) -(e^\sigma +\alpha e^\Delta) \tilde{n}_0^{(out)}  (t)+ \alpha \left( \frac{\tilde{n}_S  (t)}{ \alpha (e^{\Delta + E}  -1 )}+ F(t) \right) 
 \]
 which implies that 
 \[
 \tilde{n}_0^{(out)}(t)=\frac{e^\Delta }{ \alpha (e^{\Delta +E} -1 )(e^\Delta -1 )} \tilde{n}_S(t) . 
 \]
 \paragraph{Asymptotic behaviour of $n_k^{(out)}$ as $k \to \infty $ when $\Delta =  \Delta_c $.}
 Equality \eqref{eq:n_k inner delta_c} implies that
  \begin{equation} \label{k large N-k large delta c}
 \tilde{n}_k^{(out)} (t) \sim  \frac{ k e^{-k E } \tilde{n}_S(t) }{ \alpha (e^{\Delta + E}  -1 )}\text{ as } k \to \infty. 
 \end{equation}
 The quasi steady state approximation implies that $ n_k^{(out)} (t) \approx  \frac{ k e^{-k E } \tilde{n}_S(t) }{ \alpha (e^{\Delta + E}  -1 )}$ as $k \gg 1 $ and $N-k \gg 1 $ and for $t \approx N $. 
As in the subcritical case $\Delta < \Delta_c $, in this case we have that $\lambda (\sigma, E) = E$, hence $\lambda (\sigma, E) $ does not depend on the binding energy $\sigma$.

\subsection{Inner region: $N-k  \approx 1 $ as $N \to \infty$}
We now study the behaviour of $n_k $  in the region in which $k = N-m $ and $m \approx 1 $.
In this region $k $ is large and the dynamics is described by the equation for $n_k $ in \eqref{ODEs}, i.e.  
\begin{align}\label{eq stationary inner1}
        \frac{d n_k }{dt} &= n_S e^{- k E} - \left( e^{\sigma } + \alpha e^E + \alpha e^\Delta  \right)  n_k + \alpha e^E n_{k+1}  + \alpha e^{\Delta} n_{k-1} - \mu n_k, \quad k \in \{0, \dots, N  \}  \\
       n_{ N+1}&=0. \nonumber 
\end{align}
Notice that we are rewriting the equation for $n_N $ in \eqref{ODEs} introducing artificially the state $N+1$ in the model and imposing that $n_{N+1}=0$. This allows to have the same equation to describe the evolution of $n_k $ for $k < N $ and for $k =N $. 

Recall that $M$ changes in exponentially long times, i.e. for $t \approx e^{cN} $ for a positive constant $c$. 
This allows us to approximate $n_k$ with a quasi steady state solution $\tilde{n}_k$ 
Since we assume that $ \mu \sim e^{- b N } $ as $N \to \infty $, this quasi steady state approximation for the solution to equation \eqref{eq stationary inner1} satisfies the following system of equations
\begin{align} \label{eq stationary inner}
       0 &= \tilde{n}_S e^{- k E}  - \left( e^{\sigma } + \alpha e^E + \alpha e^\Delta  \right)   \tilde{n}_k + \alpha e^{E }  \tilde{n}_{k+1}  + \alpha e^{\Delta}  \tilde{n}_{k-1}, \quad  k \in \{0, \dots, N  \} \\
       \tilde{n}_{N+1}&=0. \nonumber
\end{align}
Notice that the condition $ \tilde{n}_{N+1} =0$ plays the role of an absorbing boundary conditions, indeed when a complex reaches the state $N+1$, it is lost. This boundary condition, together with the degradation term $\mu $, is the reason why the total mass in our system slowly decreases in time when $t \approx e^{cN } $. 
In order to analyse the solution to \eqref{eq stationary inner} we consider separately the subcritical case from the supercritical case. 
\subsubsection{Subcritical case $\Delta < \Delta_c $.}
In order to study the solution to \eqref{eq stationary inner} it is convenient to consider the following change of variables 
\begin{equation} \label{change of variable subcritical}
h_m (t) := e^{(N-m)E}  \tilde{n}_{N-m }(t) =e^{k E}  \tilde{n}_k (t)
\end{equation} 
where we recall that $N-k=m$. 
Therefore we now study the following equation for $h_m$ 
\begin{align} \label{eq inner delta <}
       0 &=  \tilde{n}_S(t)  - \left( e^{\sigma } + \alpha e^E + \alpha e^\Delta  \right)  h_m (t)+ \alpha  h_{m-1} (t) + \alpha e^{\Delta+ E} h_{m+1} (t), \quad m \in \mathbb N  \\
       h_{-1}(t)&=0 \nonumber, \\
   \lim_{m \to \infty } h_m(t) &=  A(0)  \tilde{n}_S(t)  .  \nonumber
\end{align}
We stress that the condition $h_{-1}(t)=0$ corresponds to $ \tilde{n}_{N+1}(t)=0$. Instead the condition on the limit of $h_m $ as $m \to \infty $ is a consequence of the matching with the behaviour of $ \tilde{n}_k$ in the outer region, i.e. \eqref{k large N-k large delta small}. 

Notice that $h^{(p)}_m = A(0)  n_S(t)  $ is a particular solution to equation \eqref{eq inner delta <}. 
We now study the homogeneous solutions to equation \eqref{eq inner delta <}, i.e. solutions to \eqref{eq inner delta <} without the source term.
In particular we look for homogeneous solutions of the form 
\[
h^{(h)}_m =\Theta^m, \quad \forall m \in \mathbb N.
\]
Substituting this ansatz in equation \eqref{eq inner delta <} we obtain that $\Theta $ must satisfy the following equation
\[
 \Theta^2 - \frac{ \Omega(0)}{\alpha e^{E+\Delta}} \Theta + e^{-(E+ \Delta) }  =0 .
\]
This equation has two solutions $\Theta_{1}$ and $\Theta_2$ given by 
\[
\Theta_{1} = \frac{1}{2 \alpha e^{E+\Delta} } \left( \Omega(0) -  \sqrt{\Omega (0)^2 - 4 \alpha^2 e^{E+\Delta}}  \right) \text{ and } \Theta_{2} = \frac{1}{2 \alpha e^{E+\Delta} } \left( \Omega(0) +  \sqrt{\Omega (0)^2 - 4 \alpha^2 e^{E+\Delta}}  \right) .
\]
Notice that since we expect the limit as $m \to \infty $ of $ h_m $ to be finite for every fixed time $t>0 $ then $\Theta^m $ is an admissible homogeneous solution only when $\Theta \leq 1 $. 
Moreover, since $\Delta < \Delta_c $, then we have that 
\[
 \Theta_2 = \frac{1}{2 \alpha e^{E+\Delta} } \left( \Omega(0) + \sqrt{\Omega (0)^2 - 4 \alpha^2 e^{E+\Delta}}  \right) >1. 
\]
Hence there exists a unique admissible homogeneous solution, which is 
\[
\Theta_1= \frac{1}{2 \alpha e^{E+\Delta} } \left( \Omega(0) - \sqrt{\Omega (0)^2 - 4 \alpha^2 e^{E+\Delta}}  \right) = \varphi(\sigma) e^{- (E+\Delta) } < 1.
\]

As a consequence, the solution to \eqref{eq inner delta <} is given by 
\[
h_m(t) = h_m^{(p)} (t)+  h_m^{(h)} (t)= A(0)   \tilde{n}_S (t)+ F(t) \Theta^m = A(0)  \tilde{n}_S(t)  + F(t) \varphi(\sigma)^m e^{- m (E+ \Delta) }, \quad  m \in \mathbb N \cup \{ -1\}. 
\]
Moreover, the condition $h_{-1}=0$, implies that 
\[
F(t) = - A(0)  \varphi(\sigma)  e^{- (E+\Delta) }  \tilde{n}_S(t)  . 
\]
Summarizing we have that
\[
h_m = A(0)   \tilde{n}_S  ( 1 -    \varphi(\sigma)^{m+1} e^{- (m+1) (E+ \Delta) })\quad  m \in \mathbb N \cup \{ -1\}.
\]
Hence we obtain that 
\begin{equation} \label{asymptotics of n_N sub}
 \tilde{n}_{N} (t) \sim  e^{- N E } h_0=  e^{- N E } A(0)    ( 1 -    \varphi(\sigma) e^{-  (E+ \Delta) })  \tilde{n}_S(t) \text{ as } N \to \infty. 
\end{equation}
Using the quasi steady state approximation $\tilde{n}_{N} (t) \approx n_{N} (t) $ for $t \approx N $ we deduce that 
\[
 n_{N} (t) \sim  e^{- N E } h_0=  e^{- N E } A(0)    ( 1 -    \varphi(\sigma) e^{-  (E+ \Delta) })  \tilde{n}_S(t) \text{ as } N \to \infty , \ t \approx N.
\]
Here we stress again that the asymptotic behaviour of $n_N $ as $N \to \infty $ is of the form $ F(t) e^{- N \lambda(\sigma, E) } $ where the parameter $\lambda (\sigma, E) $, in this subcritical case $\Delta < \Delta_c $, does not depend on $\sigma$  and where $F(t) $ is a suitable function of time. 
\subsubsection{Supercritical case $\Delta > \Delta_c $.}
As we did in the subcritical case (see \eqref{change of variable subcritical}) we remove the leading exponential behaviour by using the change of variables 
\[ 
\ell_k =  e^{ k \lambda (\sigma)  } \tilde{n}_k = e^{ k E - k \log(\varphi(\sigma)) } \tilde{n}_k, 
\] 
where $\lambda $ is defined as \eqref{lambda}. 
Performing the change of variable in equation \eqref{eq stationary inner}
 we deduce that $\ell_k $ must satisfy the following equation
\begin{align*}
0 &= \tilde{n}_S(t)  - \left( e^{\sigma } + \alpha e^E + \alpha e^\Delta  \right) e^{ k \log(\varphi(\sigma)) }   \ell_k(t) + \alpha e^{ (k +1)\log(\varphi(\sigma)) }  \ell_{k+1}  (t) + \alpha e^{\Delta+ E +   ( k-1) \log(\varphi(\sigma)) } \ell_{k-1}(t) \\
       0& =\ell_{N+1}(t) \nonumber. 
       \end{align*}
       Since we are analysing the behaviour in the inner region where $k \gg 1 $ we have that $\tilde{n}_S \ll e^{ k \log(\varphi(\sigma)) }   \ell_k$, $\tilde{n}_S \ll e^{\Delta+ E +   ( k-1) \log(\varphi(\sigma)) } \ell_{k-1}$, and $\tilde{n}_S \ll e^{ (k +1)\log(\varphi(\sigma)) }  \ell_{k+1} $. Therefore we can approximate the equation above as follows 
       \begin{align*}
0 & \sim  - \left( e^{\sigma } + \alpha e^E + \alpha e^\Delta  \right) \varphi(\sigma)    \ell_k (t)+ \alpha  \varphi(\sigma)^2  \ell_{k+1}(t)  + \alpha e^{\Delta+ E  } \ell_{k-1} (t), \text { as } k \to \infty \\
       0&= \ell_ {N+1}(t) \nonumber. 
       \end{align*}
       Moreover, we define $h_m = \ell_{N-m } $. Then $h_m$ satisfies the following equation 
              \begin{align*}
0 & =   - \left( e^{\sigma } + \alpha e^E + \alpha e^\Delta  \right) \varphi(\sigma)   h_m(t) + \alpha \varphi(\sigma)^2  h_{m-1} (t)  + \alpha e^{\Delta+ E  }  h_{m+1} (t), \quad m \in \mathbb N \\
       h_{-1}(t)&=0 \nonumber \\
       \lim_{m \to \infty} h_m (t)&=B A (0) \tilde{n}_S(t)
       \end{align*}
       where the last equality comes from the matching condition, indeed we have that for $k \gg 1 $ and $N-k \gg 1 $ we have that the behaviour of $\tilde{n}_k$ is given by  \eqref{k large N-k large delta large}. 
       
       We make the ansatz $h_m(t) = \Theta^m $ for every $m \in \mathbb N  $ and obtain that 
       \[
       \alpha e^{\Delta+ E  }  \Theta^2 - \Omega(0) \varphi(\sigma)    \Theta  + \alpha  (\varphi(\sigma))^2 =0.   
       \]
       We deduce that 
       \[
       \Theta_{12} = \frac{\varphi(\sigma)}{2\alpha e^{\Delta+ E  }} \left( \Omega(0) \pm \sqrt{\Omega (0)^2 - 4 \alpha^2 e^{E+\Delta}} \right) 
       \]
       In particular we obtain that 
       \[
       \Theta_1 =  \frac{\varphi(\sigma)}{2\alpha e^{\Delta+ E  }} \left( \Omega(0) - \sqrt{\Omega (0)^2 - 4 \alpha^2 e^{E+\Delta}} \right) = \varphi(\sigma)^2 e^{- (\Delta+ E ) } 
       \]
       and 
            \[
       \Theta_2 =  \frac{\varphi(\sigma)}{2\alpha e^{\Delta+ E  }} \left( \Omega(0) + \sqrt{\Omega (0)^2 - 4 \alpha^2 e^{E+\Delta}} \right) = 1
       \]
       Both these solutions are admissible as they are both smaller or equal than one. 
We deduce that 
\[
h_m(t) = F_1(t) + F_2(t) e^{ 2m \log \varphi(\sigma) - m (\Delta + E ) }, \quad \forall m  \in \mathbb N  \cup \{ -1 \}. 
\]
Matching the behaviour in the inner with the behaviour in the outer region implies that $F_1 (t)= B A(0) n_S(t) $. 
On the other hand, since $h_{-1} (t)=0$, we also have that 
\[ 
F_2(t) = - F_1  (t) e^{ -(\Delta + E)} (\varphi(\sigma))^2= - B  A(0) n_S(t)e^{- ( \Delta + E )}(\varphi(\sigma))^2, 
\] 
where we recall that $B$ is given by \eqref{B}. 
Summarizing we have that 
\[
h_m(t) = B A(0) \left( 1 - e^{ 2(m+1) \log \varphi(\sigma) - (m+1) (\Delta + E ) } \right) \tilde{n}_S (t), \quad \forall m  \in \mathbb N  \cup \{ -1 \}
\]
therefore  we have that when $t \approx N $ it holds that 
\begin{equation}\label{asymptotics of n_N super}
n_{N}(t) \sim e^{- NE + N \log (\varphi(\sigma)) } h_0 (t)= e^{- N E + N \log (\varphi(\sigma)) } A(0) B \left( 1 - ( \varphi(\sigma))^2 e^{  -  (\Delta + E ) } \right)  \tilde{n}_S (t), \text{ as } N \to \infty.
\end{equation}
Notice that the asymptotic behaviour of $n_N $ as $N \to \infty $ is of the form $ F(t) e^{- N \lambda (\sigma, E)} $ where $\lambda (\sigma,E)$ depends on $\sigma$ where $F(t) $ is a function of time that takes into account the change of mass, notice that this function changes very slowly in time, more precisely in time scales much larger than $N$. 

\subsubsection{Critical case $\Delta = \Delta_c $.}
For completeness, we now analyse the behaviour of $\tilde{n}_k $
in the inner region under the assumption that $\Delta = \Delta_c$.
To this end we make the change of variables $\ell_k := \frac{ e^{k E}}{k} \tilde{n}_k $ in equation \eqref{eq stationary inner}
 and deduce that $\ell_k $ satisfies 
\begin{align} \label{eq:1 delta_c inner}
       0 &= \tilde{n}_S (t) - \left( e^{\sigma } + \alpha e^E + \alpha e^\Delta  \right)  k \ell_k (t) + \alpha (k+1) \ell_{k+1} (t) + \alpha e^{\Delta+E} (k-1) \ell_{k-1} (t), \quad k \in \{ 0, \dots, N \}  \\
       0 &=\ell_{N+1} (t). \nonumber 
\end{align}
Since in the inner region we have that $k \gg 1 $, then $\tilde{n}_S $ is of lower order with respect to the other terms in the equation. Therefore, we approximate equation \eqref{eq:1 delta_c inner} with the following homogeneous equation 
\begin{align} \label{eq: delta_c inner homogeneous}
       0 &= - \left( e^{\sigma } + \alpha e^E + \alpha e^\Delta  \right)  k \ell_k  (t) + \alpha (k+1) \ell_{k+1} (t)  + \alpha e^{\Delta+E} (k-1) \ell_{k-1}  (t) \text{ as } k \to \infty \\
       \ell_{N+1}(t) &=0. \nonumber  
\end{align}
In order to match the behaviour in the inner region with the behaviour in the outer region it is convenient to consider the change of variable $k = N-m $ and to study $h_{m} := (N-m) \ell _{N-m} $. Notice that, due to its definition, $h_{m}$ satisfies the following equation 
\begin{align} \label{eq: delta_c inner homogeneous m}
       0 &=  - \left( e^{\sigma } + \alpha e^E + \alpha e^\Delta  \right)  h_m(t)   + \alpha h_{m-1} (t) + \alpha e^{\Delta+E} h_{m+1} (t) \\
       h_{-1}(t)&=0, \quad m \in \mathbb N \nonumber \\
       \lim_{m \to \infty } h_m (t) &= \frac{ \tilde{n}_S(t)  }{\alpha (e^{\Delta + E } -1 ) } \nonumber. 
\end{align}
Notice that the last condition comes from the matching with the outer region, i.e. by the fact that the behaviour in the outer region is given by \eqref{k large N-k large delta c}. 

Since the equation for $h_m $ is homogeneous it has solutions of the form $h_m= \theta^m $ where $\theta $ satisfies 
\[
\theta^2 - \left( \frac{ 1+ e^{\Delta + E } }{e^{\Delta + E }}\right) + \frac{1}{e^{\Delta + E }}=0 . 
\]
We deduce that both $\theta_1^m  = 1 $ and $\theta_2^m  = e^{-m (\Delta+ E) } $ are admissible solutions to \eqref{eq: delta_c inner homogeneous m}
. 
Therefore we have that the sequence $\{ h_m\}_{m} $ given by 
\[
h_m(t) = F_1(t) + F_2(t) e^{-m (\Delta+ E) },
\]
with $F_2=-F_1e^{- (\Delta+ E) } $ and with 
\[
F_1(t) = \frac{\tilde{n}_S (t)}{\alpha (e^{\Delta + E} - 1 ) },
\]
is the solution to \eqref{eq: delta_c inner homogeneous m}.
In particular we conclude that $h_0 = \frac{n_S (t) } {\alpha e^{\Delta + E} }$ and therefore 
\begin{equation} \label{asymptotics of n_N crit}
n_{N} (t) \sim \frac{\tilde{n}_S(t)}{\alpha e^{E+\Delta } } N e^{-N E } \text{ as } N \to \infty \text{ and } t \approx N.  
\end{equation}

\subsection{Discrimination property} \label{sec:discrimination}
We now study the discrimination property of the model under different assumptions on $\Delta$.
To this end we analyse the probability of response $p_{res} $ defined as in \eqref{prob response} as $N \to \infty $ in the different regimes, subcritical, critical and supercritical using the asymptotics \eqref{asymptotics of n_N sub}, \eqref{asymptotics of n_N super}, \eqref{asymptotics of n_N crit} for $n_N$. 
\paragraph{Subcritical case $\Delta < \Delta_c$.}
Using the asymptotics for $n_{N} $ in \eqref{asymptotics of n_N sub} we approximate the probability of response as follows 
\begin{align*} 
p_{res} (\sigma) = \alpha e^\Delta \int_0^\infty n_{N} (t) dt \approx \alpha e^{\Delta} e^{-NE} A(0) \overline n_S \left(1- \varphi(\sigma) e^{-(E+ \Delta )} \right) \int_0^\infty M (t) dt \text{ as } N \to \infty. 
\end{align*} 
In the approximation above, we are using that the region where $t \approx N $ gives the main contribution to the integral $\int_0^\infty n_N(t) dt $.
Moreover, notice that we  have that 
\begin{equation} \label{asympt mass}
\frac{dM }{dt} = - \mu M - \alpha e^{\Delta } n_{N} \sim - \left( \mu + \alpha e^{\Delta} e^{-NE} A(0) \overline n_S \left(1- \varphi(\sigma) e^{-(E+ \Delta )} \right) \right) M (t) \text{ as } N \to \infty. 
\end{equation}
Notice that equation \eqref{asympt mass} implies that the mass changes in exponentially large time scales.
More precisely we have that 
\begin{equation} \label{consistency}
M(t) \sim e^{-  \left( \mu + \alpha e^{\Delta} e^{-NE} A(0) \overline n_S \left(1- \varphi(\sigma) e^{-(E+ \Delta )} \right) \right) t } \text{ as } N \to \infty. 
\end{equation}
Using this approximation for $M $ we deduce that 
\begin{equation} \label{formal p res sub}
p_{res} (\sigma)  \approx \left(  1 +\frac{\mu e^{NE}}{ \alpha e^{\Delta} A(0) \overline n_S \left(1- \varphi(\sigma) e^{-(E+ \Delta )} \right) } \right)^{-1} \text{ as } N \to \infty. 
\end{equation}
Using the fact that $\mu \sim e^{- b N } $ as $N \to \infty $ we obtain that 
\[
p_{res}(\sigma)  \approx \left(  1 + C_1e^{N(E- b)} \right)^{-1} \text{ as } N \to \infty. 
\]
where $C_1:=\left[ \alpha e^{\Delta} A(0) \overline n_S \left(1- \varphi(\sigma) e^{-E- \Delta} \right) \right]^{-1}$. 
In particular we conclude that $p_N \approx 0 $ when $E>b $, $p_N \approx 1 $ as $N \to \infty $ if $E< b $. 
Notice that in this case if we plot the probability $p_{res} (\sigma) $, defined as \eqref{prob response}, as a function of $\sigma $ we do not have a sharp transition from $0 $ to $1 $ around a critical binding energy $\sigma_c$ (see Figure \ref{fig2}). The dependence on $\sigma $ is encoded only in the constant $C_1$.

\paragraph{Critical case $\Delta = \Delta_c$.}

Arguing as in the subcritical case $\Delta < \Delta_c$ we deduce that \eqref{asymptotics of n_N crit} implies 
\[
p_{res} (\sigma) \approx  \left( 1 + \frac{e^{N(E-b)}}{ N \overline n_S } \right)^{-1} \text { as } N \to \infty
\] 
We conclude that $p_{res} (\sigma) \approx 0 $ as $N \to \infty $ when $E>b $, while $p_{res} (\sigma) \approx 1 $ as $N \to \infty $ if $E \leq b $.
The situation is similar to the one that we have for the subcritical case $\Delta < \Delta_c $. 

\paragraph{Supercritical case $\Delta > \Delta_c$}

Using the approximation for $n_{N} $ as $N \to \infty $ given in \eqref{asymptotics of n_N super} we deduce that
\begin{equation}\label{formal p super}
p_{res} (\sigma)  \approx \left( 1+C_2 e^{(E -b- \log(\varphi(\sigma) ) N }\right)^{-1}
\end{equation}
where $C_2= \left[ \alpha e^\Delta B A(0) \overline n_S \left( 1- (\varphi(\sigma))^2 e^{ - (\Delta + E) } \right)  \right]^{-1} $. 
Recall that $\varphi(\sigma) $ is a function of $\sigma $, indeed 
\[ 
\varphi(\sigma)= \frac{1}{2 \alpha } \left(  e^\sigma + \alpha (e^\Delta + e^E)  -  \sqrt{(e^\sigma + \alpha (e^\Delta + e^E) )^2 - 4 \alpha^2  e^{\Delta + E} } \right). 
\]

Moreover notice that 
\[
\varphi'(\sigma)= \frac{e^\sigma}{2 \alpha } \left( 1  - \frac{e^\sigma + \alpha (e^\Delta + e^E)  }{ \sqrt{(e^\sigma + \alpha (e^\Delta + e^E) )^2 - 4 \alpha^2  e^{\Delta + E} } } \right) <0, 
\]
hence the function $\varphi$ is strictly decreasing. In particular this implies that the function $\lambda(\sigma, E)= E- \log(\varphi(\sigma)) $ is increasing in $\sigma$. 
As a consequence if we assume that 
\begin{equation} \label{conditions on b}
\lim_{\sigma \to \infty } \lambda(\sigma, E)  = E>  b > \max\{  E - \Delta  , 0 \} = \lim_{\sigma \to - \infty } \lambda(\sigma, E),    
\end{equation}
then we have that there exists $\sigma_c \in \mathbb R $ such that $\lambda(\sigma_c , E) = b $. 
Notice that if we consider $p_{res} $ as a function of $\sigma $ we see a sharp transition from $1$ to $0$, i.e. from response to non-response, around the critical binding energy $\sigma_c $. 
The region in which the transition takes place is the region in which we have that 
\[
|\lambda (\sigma, E) - \lambda (\sigma_c, E) | \lesssim \frac{1}{N }. 
\]
This region tends to $0 $ as $N \to \infty $. 
If instead we have that $\sigma $ is such that $\lambda (\sigma, E ) - \lambda (\sigma_c, E ) > \frac{1}{N }$ then we have $p_{res}(\sigma) \approx 1 $ as $N \to \infty $ and if $\lambda (\sigma_c, E) - \lambda (\sigma, E) > \frac{1}{N }$ then we have $p_{res} (\sigma) \approx 0$ as $N \to \infty$.

Notice that the discrimination property that we obtain for $\Delta > \Delta_c$ relies on the fact that we assume that ligands are degraded at a rate $\mu $ that satisfy \eqref{degradation}. If we assume that ligands are not degraded, hence $\mu =0$, the discrimination property is lost.  

\section{Approximated model on $\mathbb N$} \label{sec:half line}
In this Section we study in detail the model that we obtain taking the limit as $N \to \infty $ in \eqref{ODEs}. Analysing this model rigorously using  Laplace transform methods is easier than analysing the model with a finite number of states $N$. 
Moreover the behaviour of the model on the half line gives insights on the behaviour of the model with a finite number of states $N$ when $N$ is large. 
Indeed we prove that if $\Delta < \Delta_c$, then the solution to the problem on the half line is such that as $t \to \infty $
\[
n_k(t) \sim  A(0) \overline n_S e^{- k E}  \text{ for }  k \approx t, \ t \to \infty 
\]
We refer to Theorem \ref{theo:half line asymptotics delta subcrit} for the rigorous statement. 
If instead $\Delta > \Delta_c$, then 
\[
 n_k(t) \sim B(0)  A(0) \overline n_S e^{- k (E- \log(\varphi(\sigma))}  \text{ for }  k \lesssim t , \ t \to \infty 
\]
while 
\[
 e^{ k (E- \log(\varphi(\sigma))} n_k (t) \ll 1   \text{ for }  k \gtrsim t, \ t \to \infty .  
\]
We refer to Theorem \ref{theo:half line asymptotics delta supercrit} for the rigorous statement. 
Notice that the results that we obtain for this model  are consistent with the formal arguments in Section \ref{sec:formal}. 
Moreover notice that, when $k \approx t $ we have that $n_k(t) $ is at steady state. This is consistent with the assumption that $n_N(t)$ stabilizes in time scales of order $N$ that we make in Section \ref{sec:formal}.  

The system of ODEs that we study in this section is the one that we obtain taking the limit as $N \to \infty $ in \eqref{ODEs}, i.e. 
\begin{align} \label{ODE half line detail}
    \frac{d n_S }{dt} &= - n_S  - \frac{e^{- E} }{ 1- e^{- E}} n_S + e^{\sigma }  \sum_{k=0}^{\infty} n_k \nonumber \\
    \frac{d n_0 }{dt} &= n_S  - e^{\sigma } n_0 + \alpha e^E n_1 - \alpha e^{\Delta} n_0 \\
        \frac{d n_k }{dt} &= n_S e^{- k E} - \left( e^{\sigma } + \alpha e^E + \alpha e^\Delta  \right)  n_k + \alpha e^E n_{k+1}  + \alpha e^{\Delta} n_{k-1}   \nonumber
\end{align}
with initial datum $n(0)=(1, 0, \dots, 0)$.
Notice that  since $(n_S, n_0, \dots, n_N) $ is the vector of the probabilities of reaching  the different states of the system in the time interval $(0, t] $, we have that $M(t) = n_S (t) + \sum_{k=0}^{\infty} n_k(t) =1.$

Taking Laplace transforms in all the terms in \eqref{ODE half line detail} we obtain the following systems of equations for the vector of Laplace transforms $(\hat{n}_S (z), \hat{n}_0(z), \hat{n}_1(z), \dots, \hat{n}_{N-1}(z) )$ of $(n_S(t), n_0(t), n_1(t), \dots,n_{N-1} (t))$
\begin{align}\label{Laplace ODE half line}
  \left( z + \frac{1}{1-e^{- E}} \right)   \hat{ n}_S  &= 1+ e^{\sigma }  \sum_{k=0}^{\infty} \hat{n}_k  \nonumber \\
\left(   e^{\sigma } +z + \alpha e^{\Delta} \right) \hat {n}_0 &=  \hat{ n}_S  + \alpha e^E  \hat{ n}_1  \\
\Omega (z) \hat {n}_k &= \hat{n}_S e^{- k E} + \alpha e^E \hat{n}_{k+1}  + \alpha e^{\Delta} \hat{n}_{k-1} \quad k \in \mathbb N_0 \nonumber ,
\end{align}
where $ z \in \mathbb C$ and where we use the notation 
\begin{equation}\label{Omega}
\Omega (z):= z+ e^{\sigma }  + \alpha e^{\Delta} +  \alpha e^{E}. 
\end{equation}
We will compute now explicit solutions to \eqref{Laplace ODE half line}. As a first step we compute $\hat{n}_S$. 

\begin{lemma}
    Assume that the vector $(\hat{n}_S (z), \hat{n}_0(z), \hat{n}_1(z), \dots, \hat{n}_{N-1}(z) )$ is a solution to \eqref{Laplace ODE half line}. Then the function $\hat{n}_S $ is analytic for every $z \in \mathbb C \setminus \{ z_S, 0  \} $ where 
\begin{equation} \label{zS}
z_S := - \left(  \frac{1}{1-e^{- E}} + e^\sigma \right).
\end{equation}
More precisely, we have that
    \begin{equation}\label{n_S roof}
\hat{ n}_S(z) = \frac{z+ e^\sigma }{ z \left( z - z_S\right) }.
\end{equation}
\end{lemma}
\begin{proof}
Since in this model we have that $\frac{d M(t)}{dt}  =0$ we deduce that 
\[ 
z \left( \hat{n}_S + \sum_{k=0}^{\infty} \hat{n}_k   \right) =1.
\]
Hence the equation for $\hat{n}_S $ can be rewritten as 
\[
  \left( z + \frac{1}{1-e^{- E}} \right)   \hat{ n}_S  =  e^{\sigma }  \left( \frac{1}{z} -  \hat{ n}_S \right) +1 
\]
This implies that \eqref{n_S roof} holds.
\end{proof}

We find now a particular solution to \eqref{Laplace ODE half line}. 
To this end it is convenient to define the function $A: \mathbb C \rightarrow \mathbb C $, which is given by 
\begin{equation}\label{A}
A (z) := \frac{1}{z-z_A}=\frac{1}{\Omega (z) - \alpha \left( 1+e^{E+ \Delta} \right) } \quad z \in \mathbb C, 
\end{equation}
where 
\begin{equation}\label{pole of A}
z_A := \alpha (e^{E} -1 ) (e^{\Delta } -1 )- e^\sigma. 
\end{equation}

\begin{lemma}
Assume that the function $\hat{n}_0: \mathbb C \rightarrow \mathbb C $ is given and let $\hat{n}_S $ be given by \eqref{n_S roof}. 
We define the sequence $\{ \hat{n}_k^{(p)}\}_{k=1}^\infty $ of functions  $\hat{n}_k^{(p)}: \mathbb C \rightarrow \mathbb C $  given by
\begin{equation} \label{Laplace part}
\hat{n}_k^{(p)} (z)= A(z)e^{- k E} \hat{n}_S(z),\quad  k \in \mathbb N_0.
\end{equation}
The sequence $\{ \hat{n}_k^{(p)}\}_{k=1}^\infty $ is a particular solution to 
\begin{align} \label{Laplace difference}
\Omega (z) \hat {n}_k &= \hat{n}_S e^{- k E} + \alpha e^E \hat{n}_{k+1}  + \alpha e^{\Delta} \hat{n}_{k-1} \quad k \in \mathbb N_0,
\end{align}
\end{lemma}
\begin{proof}
    We substitute \eqref{Laplace part} in the equation for $\hat {n}_k $ and obtain that
    \[
\Omega (z) \hat {n}_k(z) - \hat{n}_S(z) e^{- k E} - \alpha e^E \hat{n}_{k+1} (z) - \alpha e^{\Delta} \hat{n}_{k-1} (z)  = \left[   A(z) \left(  \Omega (z) - \alpha (1+e^{E+\Delta} ) \right) -1  \right] e^{- k E} \hat{n}_S(z) .
    \]
    Using \eqref{A} we deduce that 
    \[ \left[   A(z) \left(  \Omega (z) - \alpha (1+e^{E+\Delta} ) \right) -1 \right] e^{- k E} \hat{n}_S(z)=0.\]
\end{proof}
Notice that for every $k \in \mathbb N_0$ the function $\hat{n}_k^{(p)}$ is analytic for every $z \in \mathbb C \setminus \{ z_A,  z_S \}   $.

We now study the homogeneous solutions of equation \eqref{Laplace difference}, i.e. we find the solutions to the following equation
\begin{equation} \label{hom}
\Omega (z) \hat {n}_k(z) = \alpha e^E \hat{n}_{k+1}(z)  + \alpha e^{\Delta} \hat{n}_{k-1}(z).
\end{equation}
\begin{lemma}
Assume that $\hat{n}_0: \mathbb C \rightarrow \mathbb C $ is given and let $\hat{n}_S $ be given by \eqref{n_S roof}. 
Then equation \eqref{hom}
admits two solutions $\theta_1^k $ and $\theta_2^k$ where 
\begin{equation} \label{teta1 and teta2}
\theta_1(z) = \frac{1}{2 \alpha e^E } \left( \Omega(z)-\sqrt{\Omega(z)^2 - 4 \alpha^2 e^{E+\Delta}} \right) \ \text{ and } \ \theta_2(z) = \frac{1}{2 \alpha e^E } \left( \Omega(z)+\sqrt{\Omega(z)^2 - 4 \alpha^2 e^{E+\Delta}} \right) 
\end{equation}
\end{lemma}
\begin{proof}
    Notice that the homogeneous equation \eqref{hom} is a linear finite difference equation of order $2.$
    Therefore it has two linearly independent solutions.
    Moreover the solutions are of the form $\theta^k$. Substituting $n_k^{(h)} = \theta^k $ in \eqref{hom} we deduce that
\begin{equation} \label{polinomio hom}
\theta^2(z) -\frac{ \Omega(z) }{\alpha e^E} \theta(z)  +  e^{\Delta - E}=0. 
\end{equation}
Hence the desired result follows. 
\end{proof}

\begin{lemma} \label{lem: theta_1 small teta_2 large}
    Let $\theta_1: \mathbb C \rightarrow \mathbb C $ and $\theta_2 : \mathbb C \rightarrow \mathbb C $ be given by \eqref{teta1 and teta2}. 
    Then we have that 
    \[
    |\theta_1(z) | \leq 1 , \quad |\theta_2 (z) | \geq 1 \quad \forall z \in \mathbb C. 
    \]
\end{lemma}
\begin{proof}
    We start by proving that $|\theta_2(z) |\geq 1$. 
    To this end it is convenient to prove that if $\rho \geq R_0$ such that $R_0 \leq 1$ and $\psi \in (0, 2 \pi)$ then it holds that
    \begin{equation} \label{ineq complex}
  |  1+ \sqrt{1+\rho e^{i \psi }} | \geq 1+ \sqrt{1- R_0}. 
    \end{equation}
    To prove this we first of all notice that $\xi:= \sqrt{1+\rho e^{i \psi }}$ is such that $\Re(\xi) \geq \sqrt{1- R_0}$. Indeed, let $\xi=\xi_1+ i \xi_2 $ for $\xi_1, \xi_2 \in \mathbb R$. This implies that $1+ \rho \cos (\psi) + i \rho \sin(\psi) =  \xi_1^2- \xi_2^2 + 2 i \xi_1 \xi_2 $. Hence we deduce that $ \rho\sin (\psi) =  2 \xi_1 \xi_2 $ and $1+ \rho \cos (\psi) = \xi_1^2- \xi_2^2$. In particular this implies that
    \[
    \xi^2_1= 1+\rho \cos(\psi) + \frac{\rho^2 (\sin(\psi))^2 }{4 \xi_1^2} \geq 1-\rho \geq 1-R_0
    \]
    and the desired conclusion follows. 
    Using the fact that $\Re(\xi) \geq \sqrt{1- R_0} $ we obtain that
    \[
     |  1+ \sqrt{1+\rho e^{i \psi }} | \geq   |  1+ \sqrt{1-R_0}+ i \xi_2 | \geq    1+ \sqrt{1-R_0} .
    \]
  
    In order to apply \eqref{ineq complex} to prove that $|\theta_2 (z)|>1$ it is convenient to use the notation $\beta:= e^{E- \Delta }$ and $H(z):= \frac{\Omega(z)}{\alpha (1+\beta)} $. We deduce that
    \[
    \theta_2(z)= \frac{1+\beta}{2 \beta} H(z)  \left( 1+ \sqrt{1- \frac{4 \beta}{(1+\beta )^2 H(z)^2}} \right).
    \]
    
    Let us assume now that $\Delta \geq E$, hence $\beta \leq 1 $. 
  In this case,
  \[ 
  \rho = \left|\frac{4 \beta}{(1+\beta )^2 H^2}\right| \leq  \frac{4 \beta}{(1+\beta )^2 \left|H\right|^2} \leq  \frac{4 \beta \alpha^2 }{ \left|\Omega\right|^2} \leq \frac{4 }{1+\frac{2}{\beta} +\frac{1}{\beta^2}} \leq 1.  
  \] 
  Hence we choose $R_0=  \frac{4 }{1+\frac{2}{\beta} +\frac{1}{\beta^2}}$ and using \eqref{ineq complex} we deduce that $|\theta_2 (z)| \geq \frac{1+ \beta }{ 2 \beta } |H| \geq  1 $. 

  Assume now that $\beta \geq 1 $ because $E \geq \Delta$. Define $\gamma = \frac{1}{\beta } \leq 1 $. Notice that 
    \begin{align*}  \theta_2(z)= \frac{1+\gamma}{2 \gamma} H(z) \left( 1+ \sqrt{1- \frac{4 \gamma}{(1+\gamma)^2 H^2(z)}} \right). 
  \end{align*}
  Therefore we can repeat the argument above used to prove that $|\theta_2(z)| \geq 1$ when $\beta \leq 1 $. 

  Let us prove now that $| \theta_1 (z) | = \left| \frac{1+\beta}{2 \beta} H  \left( 1- \sqrt{1- \frac{4 \beta}{(1+\beta )^2 H^2}} \right) \right| \leq 1 $.
  Notice that $|\theta_1(z) \theta_2(z) | =\frac{2}{1+\beta} $. Assume now that $ \beta \geq 1 $. In this case we have that $|\theta_1 \theta_2 | \leq 1 $ and $|\theta_2 | \geq 1$. As a consequence we must have $|\theta_1| \leq 1 $. 
  Assume now that $\beta \leq 1 $. Let us consider as before $\gamma =\frac{1}{\beta } \geq 1$. 
  We can now write
  \[
  \theta_1 (z)  = \frac{1+\gamma}{2 \gamma} H  \left( 1- \sqrt{1- \frac{4 \gamma}{(1+\gamma )^2 H^2}} \right) 
  \] 
  and use the argument above to deduce that $|\theta_1 (z) |\leq 1 $. 
\end{proof}

\begin{remark}
    In this paper we use the following definition of square root function in the domain $\mathbb C \setminus (- \infty , 0 )$
    \[
    \sqrt{z}:= |z|^{1/2} e^{i \psi /2} , 
    \]
    where $z= |z| e^{i \psi} $  with  $\psi \in (- \pi, \pi]$. 
\end{remark}

The solution to \eqref{Laplace ODE half line} can be written as the sum of the particular solution and the homogeneous solution. Hence we have the following statement. 
\begin{proposition} \label{prop:expression of st st}
    Let $\hat{n}_S $ be given by \eqref{n_S roof}. Let $A$ be given by \eqref{A} and let $\Omega $ be defined as \eqref{Omega}. 
    Then the sequence of functions $\{ n_k\}_{k=0}^\infty$ defined as 
    \begin{equation} \label{formula entire line}
    \hat{n}_k (z)=A(z) \hat{n}_S  e^{- k E} (1+ B(z) \varphi(z)^k ), \quad k \in \mathbb N 
    \end{equation}
with  
\[
B(z)= -  \frac{ \alpha e^E (e^\Delta -1 ) }{z+e^\sigma + \alpha e^\Delta - \alpha \varphi(z) } \text{ and }
\varphi(z) =\frac{1}{2 \alpha  } \left(  \Omega(z) -  \sqrt{\Omega(z)^2 - 4 \alpha^2  e^{\Delta + E} } \right)  
\]
   is the solution to \eqref{ODE half line}.  
\end{proposition}
\begin{proof}
    The solution to \eqref{ODE half line} is given by the linear combination of the particular solution and of the homogeneous solutions. 
    In particular, the solution must have the following form 
\begin{equation} \label{solution Laplace}
\hat{n}_k (z)= \hat{n}_k^{(p)}  (z) + C (z) \theta_1(z)^k(z) + D (z) \theta_2(z)^k= A(z) e^{- k E}  \hat{n}_S  (z) + C (z)  \theta_1(z)^k+ D(z)  \theta_2(z)^k, \quad k \in \mathbb N .  
\end{equation}
Recall that the solution to \eqref{ODE half line detail} are such that $ \hat{M}(z)= \hat{n}_S(z) + \sum_{k=0}^\infty \hat{n}_k(z) <\infty $. Since $|\theta_2(z)|\geq 1 $ for every $ z \in \mathcal C$ we must take $D(z)=0$ for every $z \in \mathbb C$. We now compute  $C(z)$. 
To this end we notice that the equation for $\hat{n}_0 $ combined with \eqref{solution Laplace} implies that 
\[ 
(e^\sigma + z + \alpha e^\Delta ) (A \hat{n}_S(z) + C(z) ) = \hat{n}_S(z)  + \alpha e^E  (A e^{- E} \hat{n}_S (z) + C(z) \theta_1). 
\]
As a consequence we deduce that
\begin{equation} \label{C} 
C(z) =\frac{\hat{n}_S(z) (1+ A (z) (\alpha - (e^\sigma + z + \alpha e^\Delta)))}{z+e^\sigma + \alpha e^\Delta - \alpha e^E \theta_1(z) } = -  \frac{ \alpha A(z) e^E (e^\Delta -1 )\hat{n}_S (z) }{z+e^\sigma + \alpha e^\Delta - \alpha e^E \theta_1(z) }. 
\end{equation}
We then obtain \eqref{formula entire line}. 
\end{proof}

\begin{lemma}
    The function $\hat{n}_k$ is analytic in $\mathbb C \setminus \left(  \{ z_S, 0  \} \cup [z_1, z_2 ] \right) $ where
    \begin{equation} \label{z_1 and z_2}
        z_1 := - e^\sigma - \alpha ( e^E + e^\Delta ), \quad  z_2:=  - e^\sigma - \alpha ( e^{E/2} - e^{\Delta/2} )^2
    \end{equation}
    and where we recall that $z_S $ is defined in \eqref{zS}. 
\end{lemma}
\begin{proof}
    Recall that $z_S$ is a pole of $\hat{n}_S $.
    Since we have that $A(z_S) \left( 1 + B(z_S) \varphi(z_S)^k  \right) \neq 0 $, then we deduce that $z_S $ is a pole of $\hat{n}_k$. 

    Notice now that the function $A$ has one pole in $z_A $ given by \eqref{pole of A}.
    However notice that $1 + B(z_A) \varphi(z_A)^k =0$. Indeed we have that 
    $\Omega (z_A)= \alpha (e^{E+\Delta} + 1 ) $ and therefore $\varphi(z_A)=1$. Moreover we also have that $B(z_A)=-1 $ hence $ B(z_A) \varphi(z_A)^k =-1 $.
    As a consequence we have that $z_A$ is not a pole of $\hat{n}_k$. 

    Moreover we can prove that if the function $B$ has a pole, then the pole is in $z_B= - e^\sigma $. However notice that $\hat{n}_S (z_B)=0$.
    As a consequence $z_B $ is not a pole of $\hat{n}_k$. 
    Finally notice that the function $\varphi$ has a branch cut when $z \in \mathbb R$ is such that $ \Omega(z)^2 - 4 \alpha^2 e^{E+ \Delta} \leq 0$. 
    Since we have that $ \Omega(z)^2 - 4 \alpha^2 e^{E+ \Delta} \leq 0$ if and only if $[z_1, z_2]$ the desired conclusion follows. 
\end{proof}

We now aim at studying the behaviour of $n_k(t) $ for large values of $t$ and large values of $k$. To this end we study the behaviour of $n_k(t)$ when $k \approx t $ as $t \rightarrow \infty$. More precisely we will study the behaviour of $n_{ \lfloor 2 \alpha e^{\frac{E+ \Delta}{2}} \theta t \rfloor } (t)$ where $\theta \geq 0$. 

\begin{theorem} \label{theo:half line asymptotics delta subcrit}
Let us define $\tau:= 2 \alpha e^{\frac{E+ \Delta }{2}} t$ and let $k:=\lfloor \theta \tau \rfloor$ for  $\theta \geq 0$.
\\  Assume that $\Delta < \Delta_c $. 
Then 
    \begin{equation} \label{asympt half subc}
    \lim_{t \to \infty } \frac{ n_{k} (t) }{ e^{-k E} } = A(0) \overline n_S.
    \end{equation}
\end{theorem}
\begin{proof}
By the inverse Laplace transform formula we have that 
\[
n_k(t)= \frac{1}{2 \pi i } e^{- k E } \int_{\delta - i \infty }^{\delta + i \infty } e^{ zt } A(z) \hat{n}_S(z) \left( 1+ B(z)  \varphi(z)^k  \right) dz   
\]
where $\delta >   0$. Moreover recall that since we assume that $\Delta < \Delta_c$ we have that  $0> z_A $. 
Notice that $n_k=\ell_k+ h_k$ where  
\[
\ell_k (t):= \frac{1}{2 \pi i } e^{- k E } \int_{\delta - i \infty }^{\delta + i \infty } e^{ zt } A(z) \hat{n}_S(z)  dz
\]
and 
\begin{equation}\label{def:h_k}
h_k (t):= \frac{1}{2 \pi i } e^{- k E } \int_{\delta - i \infty }^{\delta + i \infty } e^{ zt } A(z) \hat{n}_S(z) B(z) \varphi(z)^k   dz
\end{equation}
Let us compute $\ell_k$. To this end notice that
\begin{align*}
\ell_k(t)=     \frac{1}{2 \pi i } e^{- k E } \int_{\delta - i \infty }^{\delta + i \infty } e^{ zt } A(z) \hat{n}_S(z)  dz=   \frac{1}{2 \pi i } e^{- k E } \int_{C} e^{ zt } A(z) \hat{n}_S(z)  dz
\end{align*}
where $C$ is a  closed contour surrounding $z_A, z_S $ and $ 0 $.
Since the function $ z \mapsto  A(z) \hat{n}_S(z) $ is a meromorphic function we use Cauchy's residue theorem and deduce that 
\[ 
\ell_k(t)= e^{- k E } \left( A(0) \overline {n}_S + e^{ z_S t } A(z_S ) \frac{1}{z_S} +  e^{ z_A t } \hat{n}_S(z_A ) \right). 
\]
Notice that since $z_A , z_S <0 $ the dominant term in $ \ell_k(t)$ as $t \rightarrow \infty $ is $ e^{- k E } A(0) \overline {n}_S $. Hence
\[
\ell_{ \lfloor \tau \theta \rfloor } \left( \frac{\tau }{ 2 \alpha e^{(E+\Delta)/2} }\right)  \sim  e^{- \lfloor  \tau \theta \rfloor E } A(0) \overline {n}_S  \text{ as } \tau \to \infty. 
\]

Let us now compute $h_k$. To this end it is convenient to perform the change of variables
\[
w:= \frac{1}{2 \alpha}  e^{ - \frac{E+ \Delta}{2}} \left( z+ e^\sigma + \alpha (e^{E} + e^{\Delta} )  \right).
\]
in the integral \eqref{def:h_k}. 
This change of variables maps $0 $ in $w_0$, $z_S $ in $w_S $, $z_A $ to $w_A$, $z_1$ to $w_1 $ and $z_2 $ to $w_2 $ where  
\begin{equation} \label{poles in new variables}
w_0:=\frac{e^{ \sigma - \frac{E+ \Delta}{2}}}{2 \alpha}    + \cosh\left( \frac{E- \Delta}{2}\right), \quad  w_S := -    \frac{e^{ - \frac{E+ \Delta}{2}}}{2 \alpha(1- e^{-E})}+   \cosh\left( \frac{E- \Delta}{2}\right)
\end{equation}
and 
\[
w_A:= \cosh \left( \frac{E+ \Delta }{2}\right), \ w_1=0, \ w_2=1. 
\] 
Using the fact that 
\[
\varphi(z)= \frac{e^{(E+\Delta )/2 }}{w+\sqrt{w^2-1 }} \  \text{ and } \ B(z)= -   \frac{ 2 e^{E/2} \sinh(\Delta/2)}{ w+ \sqrt{w^2-1}- e^{(E-\Delta)/2}} 
\] 
as well as the fact that $\tau = 2 \alpha e^{(E+\Delta )/2} t$ we deduce that 
\begin{align*}
h_k(t)&=  \frac{1}{2 \pi i } e^{- k E } \int_{\delta - i \infty }^{\delta + i \infty } e^{ zt } A(z) \hat{n}_S(z) B(z) \varphi(z)^k dz\\
&= \frac{C(\alpha, E, \Delta)}{2 \pi i } e^{k \frac{\Delta - E }{2}}e^{- w_0 \tau }\int_{\tilde{\delta} - i \infty }^{\tilde{\delta} + i \infty } e^{ w \tau  } a(w) \tilde{n}_S(w) \frac{ \left(  \frac{1}{w + \sqrt{w^2-1}} \right)^k}{ w+ \sqrt{w^2-1}- e^{(E-\Delta)/2}} dw
\end{align*}
where $ \tilde{\delta} =  \frac{1}{2 \alpha}  e^{ - \frac{E+ \Delta}{2}} \left( \delta + e^\sigma + \alpha (e^{E} + e^{\Delta} )  \right)$, $
C(\alpha, E, \Delta)= -4 \alpha  e^{ \Delta/2 + E } \sinh(\Delta/2)$,
and 
\begin{equation} \label{a}
A(z)=a(w):= \frac{(w-w_A)^{-1}}{ 2 \alpha e^{(E+\Delta)/2 } }  \ \text{ and }  \ \hat{n}_S(z)= \tilde{n}_S (w): = \frac{\left(w- \cosh\left( \frac{E- \Delta }{2} \right) \right) (w-w_0)^{-1} (w-w_S)^{-1}}{2\alpha e^{(E+\Delta )/2 }} . 
\end{equation}
Notice that, as expected, the function $f$ defined as 
\[
f(w):= \frac{ e^{ w \tau  } a(w) \tilde{n}_S(w)}{ w+ \sqrt{w^2-1}- e^{(E-\Delta)/2}} \left(  \frac{1}{w + \sqrt{w^2-1}} \right)^k. 
\]
is holomorphic in $\mathbb C \setminus \left( \{ w_S, w_0, w_A   \} \cup [0, 1 ] \right) $. 
We now compute
\[
g_k (\tau):= \int_{\tilde{\delta} - i \infty }^{\tilde{\delta} + i \infty } f(w)  dw.
\]
Notice that 
\begin{align*}
g_k (\tau)=& \int_{\tilde{\delta} - i \infty }^{\tilde{\delta}  + i \infty } f(w)  dw =  \int_{\gamma} f(w)  dw  -  \sum_{j=1}^3 \int_{\gamma_j} f(w)   dw, 
\end{align*} 
where $\gamma = \cup_{j=1}^3 \gamma_j \cup ( \tilde{\delta} - i \infty , \tilde{\delta} + i \infty) $ and where 
$\gamma_1=(- r + i \infty , \tilde{\delta} + i \infty ) $, $\gamma_3 = (- r - i \infty , \tilde{\delta} - i \infty )$ and where $\gamma_2 $ is a contour that connects $- r -i\infty  $ with $-r + i \infty  $ passing by $ w_\theta := \sqrt{ 1+\theta^2 }$ in such a way that the segment $[0,1] $ remains outside $\gamma $. 
We can  compute the first integral. Indeed
\begin{align*}
 \int_{\gamma}f(w) dw=  2 \pi i \sum_{w_j \in \Gamma } \operatorname{Res}_{ w= w_j  } f(w).
\end{align*} 
where $\Gamma $ is the set of the poles of $f$ that are in the interior of the closed curve $\gamma$. 
Moreover, for $j=1,3 $ we have that 
\[   \int_{\gamma_j} f(w)  dw=0. 
\]
Therefore in order to obtain $g_k$ we only need to compute the integral of $f $ on $\gamma_2$. 
Notice that by the definition of $f$ we have that, if $k =\theta \tau $, then
\begin{align} \label{f}
f(w) &= a(w) \tilde{n}_S(w)  e^{ w \tau - k \log(w + \sqrt{w^2-1}) }\frac{1}{ w+ \sqrt{w^2-1}- e^{(E-\Delta)/2}} \nonumber \\
&=  a(w) \tilde{n}_S(w)  e^{\tau \Phi (w)  } \frac{1}{  w+ \sqrt{w^2-1}- e^{(E-\Delta)/2}} , 
\end{align} 
with \begin{equation} \label{Phi}
\Phi(w):= w - \theta  \log(w + \sqrt{w^2-1}). 
\end{equation}
Notice now that the function $\Phi $ has a minimum in $w_\theta = \sqrt{1+\theta^2}$. Indeed, $\Phi'(w_\theta )= 1- \frac{\theta }{ \sqrt{w_\theta^2-1 }} =0$. 
Moreover notice that 
\[
\Phi''(w )= \theta \frac{ w}{ (w^2-1) \sqrt{w^2-1 }}, 
\]
hence $ \Phi''(w_\theta)= \frac{\sqrt{1+\theta^2} }{ \theta^2 } > 0$. 

Notice that $\Phi(w_\theta )= \sqrt{1+ \theta^2} - \theta \log(w_\theta+\theta ) \in \mathbb R $. 
We deform the contour $\gamma_2 $ in such a way that $\gamma_2 = \gamma_2^\prime \cup \gamma_2^{\prime \prime}  \cup \gamma_2^{\prime \prime \prime} $ is such  that every $w \in \gamma_2^{\prime  } $ is of the form $w =w_\theta + i v $ for some $v \in \mathbb R$  and $\gamma_2^{\prime \prime}$ and $ \gamma_2^{\prime \prime \prime} $ connect respectively $\gamma_2^\prime $ with  $- r -i\infty  $ and with $-r + i \infty $.  
Then, using Watson's Lemma (see \cite[Section 6.4]{bender2013advanced}) we deduce that
\begin{align*}
 \int_{\gamma_2}  f(w) dw &= \int_{\gamma_2}  a(w) \tilde{n}_S(w)  e^{ \tau \Phi(w) }\frac{1}{w+ \sqrt{w^2-1}- e^{(E-\Delta)/2}} dw \\
&\sim  \int_{\gamma_2^\prime}  a(w) \tilde{n}_S(w)  e^{ \tau \Phi(w) }\frac{1}{w+ \sqrt{w^2-1}- e^{(E- \Delta)/2}} dw  \text{ as } \tau \rightarrow \infty. 
\end{align*}
By Taylor expansion we have that for $w= \sqrt{1+\theta^2 } +i  v $ it holds that
\[ 
\Phi (w) \sim \Phi (w_\theta ) - \frac{1}{2} \Phi^{\prime \prime } (w_\theta ) v^2 \ \text{ as } v \to 0. 
\]
As a consequence we deduce that
\begin{align*}
&  \int_{\gamma_2}  a(w)  \tilde{n}_S(w)  e^{ \tau \Phi(w) }\frac{1}{ w+ \sqrt{w^2-1}- e^{(E-\Delta)/2}} dw \\
&\sim a(w_\theta ) \tilde{n}_S(w_\theta ) \frac{  e^{ \tau \Phi(w_\theta ) }}{\theta+ \sqrt{1+ \theta^2 } - e^{(E-\Delta)/2}} i  \int_{- \infty }^\infty  e^{- \frac{\sqrt{1+\theta^2} }{ 2 \theta^2 } \tau v^2 } dv \\
&
= \sqrt{\pi} i  a(w_\theta ) \tilde{n}_S(w_\theta ) \frac{  \theta  \sqrt{2}}{(\theta+ \sqrt{1+ \theta^2 } - e^{(E-\Delta)/2}) (1+\theta^2)^{1/4}} \frac{ e^{ \tau \Phi(w_\theta ) }}{\sqrt{\tau}} \text{ as } \tau \to \infty. 
\end{align*}

We conclude that 
\[
g_{\tau \theta } (\tau)  \sim \frac{1}{2 \pi i }\sum_{w_i \in \Gamma } \operatorname{Res}_{ w= w_i } f(w)-  \frac{    \sqrt{2\pi} i \theta a(w_\theta ) \tilde{n}_S(w_\theta )}{(\theta+ \sqrt{1+ \theta^2 } - e^{(E-\Delta)/2}) (1+\theta^2)^{1/4}} \frac{ e^{ \tau \Phi(w_\theta ) }}{\sqrt{\tau}} \text{ as } \tau \to \infty. 
\]
Since we have that $h_{\tau \theta } (\tau) =  C(\alpha, E, \Delta) e^{  \frac{\theta \tau (\Delta - E) }{2}} e^{- \tau w_0 } g_{\tau \theta } (\tau) $ we deduce that
\begin{align*}
    h_{\theta \tau } (t)  &=  C(\alpha, E, \Delta) e^{  \frac{\theta \tau (\Delta - E) }{2}} e^{- \tau w_0 }g_{\tau \theta}(\tau) \\
    & \sim  C(\alpha, E, \Delta) e^{ \frac{ \theta \tau (\Delta - E) }{2}  - \tau w_0 } \left[ \sum_{ w_i \in \Gamma } \operatorname{Res}_{w=w_i} f(w)  - \frac{  \theta     a(w_\theta ) \tilde{n}_S(w_\theta )}{\sqrt{2 \pi }(\theta+ \sqrt{1+ \theta^2 } - e^{(E-\Delta)/2}) (1+\theta^2)^{1/4}} \frac{ e^{ \tau \Phi(w_\theta ) }}{\sqrt{\tau}}\right] \\
    &= e^{- \tau \theta E - \tau w_0 }C(\alpha, E, \Delta) \times\\
   & \times \left[ \sum_{ w_i \in \Gamma } \operatorname{Res}_{w=w_i} f(w) e^{\theta (E+\Delta)/2 } - \frac{  \theta     a(w_\theta ) \tilde{n}_S(w_\theta )   
   e^{ \tau \Phi(w_\theta ) + \theta(E+\Delta)/2} 
   }{\sqrt{2 \pi \tau }(\theta+ \sqrt{1+ \theta^2 } - e^{(E-\Delta)/2}) (1+\theta^2)^{1/4}}  \right]\text{ as } \tau \to \infty. 
\end{align*}
In order to study the behaviour of $n_{ \lfloor \tau \theta \rfloor } = \ell_{\lfloor \tau \theta \rfloor } + h_{\lfloor \tau \theta \rfloor} $ we use the fact that $ h_{\lfloor \tau \theta \rfloor} \sim h_{ \tau \theta }$ as $\tau \to \infty$ and we compare  the leading term of $h_{\tau \theta } $ with the leading term of $\ell_{\tau \theta } $ which is 
\begin{equation} \label{leading sub}
e^{- \tau \theta E } A(0) \overline {n}_S =e^{- \tau \theta E - \tau w_0 } e^{ \tau w_0 } A(0) \overline {n}_S .
\end{equation}
Now notice that, since $\Delta < \Delta_c $, it holds that
\begin{equation}\label{comparison 1.a}
w_0= \max\left\{ w_0 , \Phi(w_A) + \theta\frac{E+ \Delta }{2}, \Phi(w_0) + \theta\frac{E+ \Delta }{2}, \Phi(w_\theta) + \theta \frac{E+ \Delta }{2}  \right\}.
\end{equation}
Indeed, we have that 
\[
e^{- (E+ \Delta) /2} \left(w_0 + \sqrt{w_0^2 -1 } \right) = \frac{1}{\varphi(\sigma) } \geq 1, 
\]
which implies that 
\[ 
\Phi(w_0) + \theta \frac{E+ \Delta }{2 }  - w_0 = \theta \left( \frac{E+ \Delta }{2 } - \log\left(w_0 + \sqrt{w_0^2 -1 } \right) \right)  <0. 
\]
Therefore $w_0 > \Phi(w_0) +  \theta \frac{E+ \Delta }{2 } $ for every $\theta $. 
Moreover let us define the function $\Psi_M$ as 
\[ 
\Psi_M (\theta ):= \Phi(w_\theta) + \theta \frac{E+ \Delta }{2}
\]
Notice that this function has a maximum in $\theta_M := \sinh \left( \frac{\Delta + E }{2}  \right) $. 
Moreover we have that $\Psi_M(\theta_M )= \cosh \left( \frac{\Delta + E }{2}  \right) $. 
Now notice that $w_0- \cosh \left( \frac{\Delta + E }{2}  \right) =e^\sigma - \alpha (e^E-1) (e^\Delta -1 )>0  $ because $\Delta < \Delta_c$. 
Hence $w_0 > \Phi(w_\theta) + \theta \frac{E+ \Delta }{2}$ for every $\theta>0$. To prove \eqref{comparison 1.a} we only need to prove that $w_0 > \Phi(w_A) + \theta \frac{E+ \Delta }{2}$ for every $\theta>0$. This follows by the fact that
\[
\Phi(w_A)= w_A- \theta \frac{E+\Delta }{2}
\]
Hence $ \Phi(w_A)+ \theta \frac{E+\Delta }{2}=w_A < w_0$. 

Recall that $\Gamma \subset \{ w_A, w_S, w_0 \} $ is the set of the poles of the function $f$ defined as in \eqref{f}. The fact that equality \eqref{comparison 1.a} holds implies that the dominant term in $ n_{\tau \theta } = \ell_{\tau \theta } + h_{\tau \theta } $ is the leading term of $\ell_{\tau \theta } $ which is \eqref{leading sub}. 
This implies that \eqref{asympt half subc} holds. 
\end{proof}

\begin{theorem} \label{theo:half line asymptotics delta supercrit}
Let us define $\tau:= 2 \alpha e^{\frac{E+ \Delta }{2}} t$ and let $k:=\lfloor \theta \tau \rfloor$ for  $\theta \geq 0$.
    Assume that $\Delta > \Delta_c $. 
    \begin{enumerate} 
\item If $\theta $ is such that $w_\theta < w_S $, then we have that  
     \begin{equation}\label{asympt half superc thetha small}
    \lim_{t \to \infty } \frac{ n_{k} (t) }{ e^{k( \log(\varphi(\sigma))- E ) } } = A(0) B(0) \overline n_S  .
    \end{equation}
 \item    If instead $\theta $ is such that  $w_\theta  >w_S  $, then 
     \begin{equation}\label{asympt half superc thetha large}
    \lim_{t \to \infty } \frac{ n_{k} (t) }{ e^{k( \log(\varphi(\sqrt{1+\theta^2 }))- E ) } } = 0.
    \end{equation}
    \end{enumerate}
\end{theorem}
\begin{proof}
As in the proof of Theorem \ref{theo:half line asymptotics delta subcrit} we apply the inverse Laplace transform formula to deduce that
\[
n_k(t)= \frac{1}{2 \pi i } e^{- k E } \int_{\delta - i \infty }^{\delta + i \infty } e^{ zt } A(z) \hat{n}_S(z) \left( 1+ B(z)  \varphi(z)^k  \right) dz   
\]
where $z_A > \delta > 0 $. 
As in the proof of Theorem \ref{theo:half line asymptotics delta subcrit} $n_k=\ell_k+ h_k$ where  
\begin{equation} \label{lk}
\ell_k (t):= \frac{1}{2 \pi i } e^{- k E } \int_{\delta - i \infty }^{\delta + i \infty } e^{ zt } A(z) \hat{n}_S(z)  dz = e^{- k E } \left( A(0) \overline {n}_S + e^{ z_S t } A(z_S ) \frac{1}{z_S} \right)
\end{equation}
and 
\begin{equation}\label{def:h_k2}
h_k (t):= \frac{1}{2 \pi i } e^{- k E } \int_{\delta - i \infty }^{\delta + i \infty } e^{ zt } A(z) \hat{n}_S(z) B(z) \varphi(z)^k   dz
\end{equation}
Notice that since $ z_S <0 $ the dominant term in $ \ell_k(t)$ as $t \rightarrow \infty $ is $ e^{- k E } A(0) \overline {n}_S $. Hence
\[
\ell_{\tau \theta } \left( \frac{\tau }{ 2 \alpha e^{(E+\Delta)/2} }\right)  \sim  e^{- \tau \theta E - \tau w_0 } e^{ \tau w_0 } A(0) \overline {n}_S  \text{ as } \tau \to \infty. 
\]
Moreover 
\begin{align*}
h_k(t)&=   \frac{C(\alpha, E, \Delta)}{2 \pi i } e^{k \frac{\Delta - E }{2}}e^{- w_0 \tau }g_k(\tau ) 
\end{align*}
where $g_k$ and $C(\alpha, E, \Delta)$  are defined as in the proof of Theorem \ref{theo:half line asymptotics delta subcrit}. 
With the same arguments as the ones used in the proof of Theorem \ref{theo:half line asymptotics delta subcrit} we deduce that 
\[
g_{\tau \theta } (\tau)  \sim  \int_{\gamma}f(w) dw -  \frac{    \sqrt{2\pi} i \theta a(w_\theta ) \tilde{n}_S(w_\theta )}{(\theta+ \sqrt{1+ \theta^2 } - e^{(E-\Delta)/2}) (1+\theta^2)^{1/4}} \frac{ e^{ \tau \Phi(w_\theta ) }}{\sqrt{\tau}} \text{ as } \tau \to \infty. 
\]
where $\gamma  $ is as in the proof of Theorem \ref{theo:half line asymptotics delta subcrit}. 
In order to compute $\int_{\gamma}f(w) dw$ we need to distinguish between the following three cases. 
\begin{enumerate} 
\item $w_0 > w_\theta > w_S  $, 
\item $ w_0 > w_S > w_\theta $, 
\item $ w_\theta > w_0> w_S $.
\end{enumerate}

\paragraph{Case 1: $w_0 > w_\theta > w_S  $.}
Assume that 1. holds, hence that $\theta $ is such that $w_S < w_\theta < w_0 $.  
Then we have that 
\begin{align*}
 \int_{\gamma}f(w) dw=  2 \pi i  \operatorname{Res}_{w=w_0} f(w).
\end{align*} 
Now notice that
\[
\operatorname{Res}_{w=w_0} f(w) = \operatorname{Res}_{w=w_0} \tilde{n}_S (w)  a(w_0) b(w_0) e^{\tau \Phi (w_0) } 
\]
where we are using the notation 
\[ 
b(w):= \frac{1}{w+\sqrt{w^2-1}-e^{(E-\Delta)/2}}
\] 
and where the function $a$ is given by \eqref{a} and the function $\Phi$ by \eqref{Phi}. 
Therefore 

\[
g_{\tau \theta } (\tau)  \sim \operatorname{Res}_{w=w_0} \tilde{n}_S (w)  a(w_0) b(w_0) e^{\tau \Phi (w_0) }  -  \frac{    \sqrt{2\pi} i \theta a(w_\theta ) \tilde{n}_S(w_\theta )}{(\theta+ \sqrt{1+ \theta^2 } - e^{(E-\Delta)/2}) (1+\theta^2)^{1/4}} \frac{ e^{ \tau \Phi(w_\theta ) }}{\sqrt{\tau}} \text{ as } \tau \to \infty. 
\]
Hence
\begin{align*}
        h_{\theta \tau } (t) \sim &C(\alpha, E, \Delta)e^{ - \theta \tau E - \tau w_0 } [   b(w_0)   a(w_0) e^{ \frac{ \theta \tau (\Delta +E) }{2}}  e^{\tau   \Phi(w_0)} \operatorname{Res}_{w=w_0} \tilde{n}_S (w) \\
        &- \frac{   \theta  b(w_\theta)   a(w_\theta ) \tilde{n}_S(w_\theta )  
        e^{ \tau \Phi(w_\theta ) + \tau\theta  \frac{E+ \Delta}{2}}
        }{\sqrt{2 \pi \tau } (1+\theta^2)^{1/4}} ]   \text{ as } \tau \to \infty. 
\end{align*}

In order to study the behaviour of  $n_{\theta \tau}(t)$ as $\tau \to \infty $ we need to identify the leading order term in $\ell_{\theta \tau}+ h_{\theta \tau}$.
Since $\Delta > \Delta_c $ and $w_0 > w_\theta  $, we have that
\[  
\Phi(w_0) + \theta\frac{E+ \Delta }{2}=\max\left\{ w_0 , \Phi(w_0) + \theta\frac{E+ \Delta }{2}, \Phi(w_\theta) + \theta \frac{E+ \Delta }{2}  \right\}. 
\] 
Indeed, the fact that $\Delta > \Delta_c $ implies that 
\[
\Phi(w_0) + \theta\frac{E+ \Delta }{2} >w_0 > w_S. 
\]
Moreover notice that by the definition of $w_\theta $ we have that  $\Phi(w_0) > \Phi(w_\theta)$. Hence the desired conclusion follows. We conclude that in this case 
\begin{equation} \label{asympt n 1b}
    n_{\theta \tau} (t) \sim C(\alpha, \Delta, E)   b(w_0) a(w_0) \operatorname{Res}_{w=w_0} \tilde{n}_S (w)e^{ \frac{ \theta \tau (\Delta + E) }{2}}  e^{\tau   \Phi(w_0)}  \text{ as } \tau \to \infty. 
\end{equation}
Therefore using that 
\[ 
C(\alpha, \Delta, E)   b(w_0) a(w_0) \operatorname{Res}_{w=w_0} \tilde{n}_S (w)e^{ \frac{ \theta \tau (\Delta + E) }{2}}  e^{\tau   \Phi(w_0)} = -  A(0) B(0) \overline{n}_S
\]
we deduce that \eqref{asympt half superc thetha small} holds. 

\paragraph{Case 2: $w_0 > w_S > w_\theta   $. }
The main difference between case $1$ and case $2$ is the value of the integral $\int_\gamma f(w( dw ) $. 
Indeed, when  $\theta $ is such that $w_S > w_\theta $ we have that 
\begin{align*}
 \int_{\gamma} f(w) dw=  2 \pi i \left( \operatorname{Res}_{w=w_0} f(w)+ \operatorname{Res}_{w=w_S} f(w) \right) . 
\end{align*}
As a consequence, computations similar to the ones made for case 1 imply that 
\begin{align*}
        h_{\theta \tau } (t)\sim &e^{ - \theta \tau E - \tau w_0 } [   b(w_0)  \operatorname{Res}_{w=w_0} \tilde{n}_S (w) e^{ \frac{ \theta \tau (\Delta - E) }{2}}  e^{\tau   \Phi(w_0)} +b(w_S) \operatorname{Res}_{n=n_S} \left(\tilde{n}_S (w ) \right) e^{ \frac{ \theta \tau (\Delta - E) }{2}}  e^{\tau   \Phi(w_S)} \\
        &
        - \frac{  \theta  b(w_\theta)   a(w_\theta ) \tilde{n}_S(w_\theta )}{\sqrt{2 \pi } (1+\theta^2)^{1/4}} \frac{ e^{ \tau \Phi(w_\theta ) + \tau\theta  \frac{E+ \Delta}{2}}}{\sqrt{\tau}}].
\end{align*}

In order to find the leading order term in $\ell_{\theta \tau } + h_{\theta \tau }$
 we prove that
\begin{equation}\label{comparison 2.b}
\Phi(w_0) + \theta\frac{E+ \Delta }{2}=\max\left\{ w_0 , \Phi(w_0) + \theta\frac{E+ \Delta }{2}, \Phi(w_\theta) + \theta \frac{E+ \Delta }{2}, \Phi(w_S) + \theta\frac{E+ \Delta }{2}  \right\}.
\end{equation}
Indeed, the fact that $\Delta > \Delta_c$ implies that 
\[
\Phi(w_0) +\theta\frac{E+ \Delta }{2} >w_0 > w_S > \Phi(w_s) \text{ and }  \Phi(w_0)> \Phi(w_\theta).
\]
Therefore the asymptotics \eqref{asympt half superc thetha small} follow. 

\paragraph{Case 3: $w_\theta  > w_0 > w_S   $. }
When  $\theta $ is such that $ w_\theta  > w_0 > w_S  $ we have that 
\begin{align*}
 \int_{\gamma} f(w) dw= 0. 
\end{align*}
As a consequence, 
\begin{align*}
        h_{\theta \tau } (t)\sim 
        - \frac{ C(\alpha, E, \Delta)  \theta  b(w_\theta)   a(w_\theta ) \tilde{n}_S(w_\theta )}{\sqrt{2 \pi } (1+\theta^2)^{1/4}} \frac{ e^{ \tau \Phi(w_\theta ) + \tau\theta  \frac{E+ \Delta}{2}}}{\sqrt{\tau}} \text{ as } \tau \to \infty. 
\end{align*}

Notice that the fact that $ w_\theta > w_0 > w_S  $ and $\Delta > \Delta_c$ imply that
\[ 
\Phi(w_\theta) + \theta \frac{E+ \Delta }{2} =\max\left\{ w_0 , \Phi(w_\theta) + \theta \frac{E+ \Delta }{2} \right\}. 
\]
Hence
\[ 
   n_{\theta \tau } (t)\sim 
        - \frac{ C(\alpha, E, \Delta)  \theta  b(w_\theta)   a(w_\theta ) \tilde{n}_S(w_\theta )}{\sqrt{2 \pi } (1+\theta^2)^{1/4}} \frac{ e^{ \tau \Phi(w_\theta ) + \tau\theta  \frac{E+ \Delta}{2}}}{\sqrt{\tau}} \text{ as } \tau \to \infty
        \text{ as } \tau \to \infty
\]
Therefore 
\eqref{asympt half superc thetha large} follows. 

\end{proof}

\begin{theorem} \label{theo:half line asymptotics delta crit}
Let us define $\tau:= 2 \alpha e^{\frac{E+ \Delta }{2}} t$ and let $k:=\lfloor \theta \tau \rfloor$ for  $\theta \geq 0$.
\\  Assume that $\Delta = \Delta_c $. 
Then 
\begin{enumerate}
\item if $\theta $ is such that $w_\theta < w_S $, then we have that  
     \begin{equation}\label{asympt half crit thetha small}
    \lim_{t \to \infty } \frac{ n_{k} (t) }{ e^{- k E } } =  \overline n_S \left( 1- \frac{1}{2\alpha e^E}\right).
    \end{equation}
 \item    If instead $\theta $ is such that  $w_\theta  >w_S  $, then 
     \begin{equation}\label{asympt half crit thetha large}
    \lim_{t \to \infty } \frac{ n_{k} (t) }{ e^{k( \log(\varphi(w_\theta ))- E ) } } = 0.
    \end{equation}
    \end{enumerate}
\end{theorem}
\begin{proof}
Notice that if $\Delta = \Delta_c $ we have that $z_A=z_0$ and as a consequence we have that 
\[
w_0=w_A= \cosh(E). 
\]
As in the proof of Theorem \ref{theo:half line asymptotics delta subcrit} and Theorem \ref{theo:half line asymptotics delta supercrit} we deduce that the function $\ell_k $ defined as in \eqref{lk} is given by 
\[
\ell_{\tau \theta }(t)\sim  e^{- \tau \theta  E - w_0 \tau } e^{ w_0 \tau }  \overline n_S \text{ as } \tau \to \infty.  
\]
In order to obtain the asymptotics of $n_{k} $ we have to compute the asymptotics of the function $h_{\theta \tau} (\tau)=   \frac{C(\alpha, E, \Delta)}{2 \pi i } e^{- w_0 \tau }g_{\tau \theta} (\tau )  $ where 
\begin{align*}
g_{\tau \theta} (\tau)= \int_{\gamma} f(w)  dw  - \int_{\gamma_2} f(w)   dw
\end{align*} 
where $\gamma $ and $\gamma_2$ are defined as in Theorem \ref{theo:half line asymptotics delta crit} and $f$ is given by \eqref{f}. 

When $w_0 > w_\theta > w_S  $ then 
\[ 
\int_{\gamma} f(w)  dw =  2 \pi i \left[ \operatorname{Res}_{ w =w_0} f(w)   \right] = 2 \pi i e^{\tau \Phi_0} \frac{b(w_0)  \overline n_S }{(2 \alpha)^2 e^{2E} } = 2 \pi i e^{\tau (w_0- \theta E)} \frac{b(w_0)  \overline n_S }{(2 \alpha)^2 e^{2E} }. 
\]
where we are using the fact that $\Phi(w_0)=w_0 - \theta E.$
This together with detailed asymptotics of  $\int_{\gamma_2} f(w) dw$ as $\tau \rightarrow \infty $ implies that as $\tau \to \infty $
\begin{align*}
g_{\tau \theta } \sim 
2 \pi i e^{\tau (w_0- \theta E)} \frac{b(w_0)  \overline n_S }{(2 \alpha)^2 e^{2E} }  - \frac{  \theta  \sqrt{2 \pi } i  a(w_\theta ) \tilde{n}_S(w_\theta )}{(\theta+ \sqrt{1+ \theta^2 } - 1) (1+\theta^2)^{1/4}} \frac{ e^{ \tau \Phi(w_\theta ) }}{\sqrt{\tau}}. 
\end{align*}
Moreover notice that
\[ 
w_0= \max_{\theta >0} \left(\Phi(w_\theta ) + \theta E \right). 
\]
and we recall that the maximum of the function $\Psi_M =  \Phi(w_\theta ) + \theta E$ is attained  at $\theta_M = \sinh(E)$, hence when $w_\theta= \sqrt{1+ \theta_M^2 } = \cosh(E)= w_0$. 
Therefore when $w_0> w_\theta $ we have that 
\begin{align*}
g_{\tau \theta } \sim 
2 \pi i e^{\tau (w_0- \theta E)} \frac{b(w_0)  \overline n_S }{(2 \alpha)^2 e^{2E} } =2 \pi i e^{\tau (w_0- \theta E)} \frac{  \overline n_S }{(2 \alpha)^2 e^{2E} (e^E-1)} . 
\end{align*}
As a consequence we deduce that 
\begin{equation} \label{asympt crit}
n_{\tau \theta } (t) \sim e^{- \tau \theta  E }\overline n_S  \left( 1- \frac{1 }{2 \alpha e^{E} } \right) \text{ as } \tau \to \infty,
\end{equation}
where we used the fact that $C(\alpha, E, E)=  2 \alpha (e^{2E} - e^E )$. 
When $w_0 = w_\theta > w_S  $ we obtain the same result. 
When $ w_0 > w_S > w_\theta $, then we have that
\[
\int_{\gamma} f(w)  dw = 2 \pi i \left[ \operatorname{Res}_{ w =w_0} f(w)  + \operatorname{Res}_{ w =w_S} f(w)  \right] 
\]
and we can prove that \eqref{asympt crit} holds also in this case. 

Finally when $ w_\theta > w_0> w_S $, then 
\[
\int_{\gamma} f(w)  dw = 0. 
\]
This implies \eqref{asympt half crit thetha large}. 
\end{proof}

\section{Analysis of the model with a finite number of kinetic proofreading steps} \label{sec:rigorous}
In this section we study in detail, using Laplace transforms, the model described by \eqref{ODEs} when the number of kinetic proofreading steps $N$ is large, but finite. We  recover the same results obtained using matched asymptotics expansions in Section \ref{sec:formal}.

Performing the Laplace transform to all the terms in equation \eqref{ODEs} we obtain the following equation for the vector of the Laplace transforms $(\hat{n}_S, \hat{n}_0, \dots, \hat{n}_{N}) $ of the concentrations $(n_S, n_0, \dots, n_{N}) $
\begin{align} \label{Laplace eq}
    \left( z + \frac{e^{- N E}-1}{1- e^{-E}} + \mu  + e^\sigma \right) \hat{n}_S(z)&= 1+ e^\sigma \hat{M}(z) \nonumber \\ 
    \left( z + e^\sigma + \alpha e^\Delta + \mu \right) \hat{n}_0 (z)&=  \hat{n}_S(z) + \alpha e^E \hat{n}_1 (z)\nonumber \\
    \Omega(z) \hat{n}_k(z)&=  e^{- k E} \hat{n}_S (z)+ \alpha e^E\hat{n}_{k+1} (z)+ \alpha e^\Delta \hat{n}_{k-1}(z), \ k \in \{1, \dots, N\} \\
    n_{N+1}(z)&=0 \nonumber
\end{align}
where $\Omega(z) $ is defined as \eqref{Omega} and where 
\[
\hat{M}(z)= \sum_{k=0}^{N } \hat{n}_k (z)  + \hat{n}_S(z) .  
\]
Since the probability of response is given by 
\[
p_{res}(\sigma) = \alpha e^\Delta \hat{n}_N(0)
\]
we aim at studying the behaviour of $\hat{n}_N (0) $ as $N \to \infty$. 
Using computations that are similar to the ones of Proposition \ref{prop:expression of st st} we prove the following Lemma. 
\begin{lemma} \label{lem:exp stst N}
Assume that $\hat{n}_S(0) \in \mathbb R_* $ and $\hat{n}_0 (0)$ are given.  
Let 
$\{ \hat{n}_k (0)  \}_{k =1}^N $ be the sequence of $\hat{n}_k(0)  $ defined as  
   \begin{equation} \label{eq nk rig}
\hat{n}_k (0)  = e^{- k E } [A(0)  \hat{n}_S(0)  + F_1(N) \varphi(\sigma)^k +F_2 (N) \varphi_2^k ], \quad  1 \leq k \leq N 
\end{equation}
where 
\[ 
\varphi_2= \frac{1}{2 \alpha } \left( \Omega(0) + \sqrt{\Omega (0)^2 - 4 \alpha^2 e^{E+\Delta}}  \right) 
\]
and where 
\begin{equation} \label{F}
F_1(N):= \frac{A(0) \hat{n}_S(0) + (\hat{n}_0 (0)- A(0) \hat{n}_S (0)) \varphi_2^{N+1} }{\varphi_2^{N+1}  -\varphi(\sigma )^{N+1} }, \quad F_2(N):= \frac{A(0) \hat{n}_S(0) + (\hat{n}_0(0) - A(0) \hat{n}_S (0)) \varphi(\sigma)^{N+1} }{\varphi(\sigma) ^{N+1}  -\varphi_2^{N+1} }. 
\end{equation}
The vector $ ( \hat{n}_S(0), \hat{n}_0(0), \hat{n}_1(0) , \dots \hat{n} _{N} (0) )  $ satisfies 
\begin{align*}
    \Omega(0) \hat{n}_k(0)&=  e^{- k E} \hat{n}_S (0)+ \alpha e^E\hat{n}_{k+1} (0)+ \alpha e^\Delta \hat{n}_{k-1}(0), \ k \in \{1, \dots, N\} \\
    n_{N+1}(0)&=0 
\end{align*}
\end{lemma}
As a consequence, enforcing that $\hat{n}_S(0) $ satisfies $ \left(  \frac{e^{- N E}-1}{1- e^{-E}} + \mu  + e^\sigma \right) \hat{n}_S(0)= 1+ e^\sigma \hat{M}(0) $ we can prove the following Lemma. 
\begin{lemma}\label{lem:exp stst N complete}
Assume that 
$\hat{n}_0(0) \in \mathbb R_* $ is given by 
\[
\hat{n}_0 (0) =G_N  \hat{n}_S(0)
\]
with 
\[
G_N = \frac{1+ \alpha A +\alpha A \left( \frac{ \varphi(\sigma) - \varphi_2 }{ \varphi_2^{N+1} - (\varphi(\sigma))^{N+1}}- \Psi_N \right) }{ e^\sigma + \alpha e^\Delta + \mu -\alpha \Psi_N} \ \text{ where } \ \Psi_N = \frac{\varphi_2^{N+1} \varphi (\sigma)-(\varphi(\sigma))^{N+1} \varphi_2  }{ \varphi_2^{N+1} - (\varphi(\sigma))^{N+1}}
\]
Assume that $\hat{n}_S(0) \in \mathbb R_* $ is given by  
\[ 
\hat{n}_S (0)= \frac{e^{- N E}-1}{1- e^{-E}} + \mu + e^\sigma C_N
\]
where 
\[
C_N:= \left( 1- A(0)  \frac{e^{- N E}-1}{1- e^{-E}} -\frac{ A(0) + (G_N- A(0) ) \varphi(\sigma)^N }{\varphi_2^{N+1} - \varphi(\sigma)^{N+1}} \frac{\varphi(\sigma)^{- N } -1 }{1- \varphi(\sigma) }- \frac{ A(0) + (G_N- A(0) ) \varphi_2^N }{\varphi(\sigma)^{N+1} - \varphi_2^{N+1}}  \frac{\varphi_2^{- N } -1 }{1- \varphi_2}\right). 
\]
Assume that the vector $( \hat{n}_0 (0) , \hat{n}_1 (0), \dots, \hat{n}_N (0) ) $ is defined as in Lemma \ref{lem:exp stst N}. Then $( \hat{n}_S (0) , \hat{n}_0 (0), \hat{n}_1 (0), \dots, \hat{n}_N (0) ) $ is the unique solution to 
\begin{align*}
    \left(  \frac{e^{- N E}-1}{1- e^{-E}} + \mu  + e^\sigma \right) \hat{n}_S(0)&= 1+ e^\sigma \hat{M}(0) \\
    \left( e^\sigma + \alpha e^\Delta + \mu \right) \hat{n}_0 (0)&=  \hat{n}_S(0) + \alpha e^E \hat{n}_1 (0) \\
    \Omega(0) \hat{n}_k(0)&=  e^{- k E} \hat{n}_S (0)+ \alpha e^E\hat{n}_{k+1} (0)+ \alpha0e^\Delta \hat{n}_{k-1}(0), \ k \in \{1, \dots, N\} \\
    n_{N+1}(0)&=0 
\end{align*}
\end{lemma}

\begin{lemma} \label{lem:asympt F1}
Let $F_1(N) $ and $F_2(N) $ be given by \eqref{F}. Then we have \[ F_1(N) \sim \hat{n}_0 (0) - A(0)\hat{n}_S(0) \  \text{ as } N \to \infty \] and
\[
F_2(N) \sim - \frac{ A(0) \hat{n}_S(0) + (\hat{n_0} (0) - A(0)\hat{n}_S(0) ) \varphi(\sigma)^{N+1}}{\varphi_2^{N+1}} \ \text{ as } N \to \infty. 
\]   
\end{lemma}
\begin{proof}
In order to prove this Lemma it is enough to notice that Lemma \ref{lem: theta_1 small teta_2 large} implies that $\varphi(\sigma) < \varphi_2$.
\end{proof}
\begin{proposition}
Assume that $\hat{n}_N (0) $ is as in Lemma \ref{lem:exp stst N complete}. Then we have that 
\begin{itemize}
\item if $\Delta < \Delta_c $, then 
\[
 \hat{n}_N (0) \sim \overline{n}_S \hat{M}(0) e^{- N E }  A(0) \left( 1- \frac{1}{\varphi_2 }\right)   \text{ as } N \to \infty. 
 \]   
 \item If $\Delta >\Delta_c $, then 
 \[
      \hat{n}_N (0) \sim \overline{n}_S \hat{M}(0) e^{- N [E + \log(\varphi(\sigma) ) ] }   (G- A(0) ) \left( 1- \frac{\varphi(\sigma ) }{ \varphi_2}\right)   \text{ as } N \to \infty
 \]
where we recall that $G $ is given by \eqref{lem:exp stst N complete}. 
 \end{itemize}

\end{proposition}
\begin{proof}
 
Lemma \ref{lem:asympt F1} together with Lemma \ref{lem:exp stst N complete} imply that
\[
\hat{n}_1 (0) =e^{-E } ( A(0) \hat{n}_S (0)+ F_1 \varphi(\sigma) + F_2 \varphi_2) \sim  e^{-E } ( A(0) \hat{n}_S (0) + (\hat{n}_0(0)  - A(0) \hat{n}_S )\varphi(\sigma) ). 
\]
Using the asymptotics for $\hat{n}_1 $ we deduce that 
\[
\hat{n}_0 (0)= \frac{\hat{n}_S + \alpha e^E \hat{n}_1(0)}{\mu + \alpha e^\Delta + e^\sigma } \sim \frac{\hat{n}_S + \alpha ( A(0) \hat{n}_S + (\hat{n}_0 - A(0) \hat{n}_S )\varphi(\sigma) ) }{ \alpha e^\Delta + e^\sigma } 
\]
Hence $\hat{n}_0(0) \sim G \hat{n}_S $ where 
\begin{equation} \label{G} 
G= \frac{1+ \alpha A(0) (1- \varphi(\sigma) )}{ e^\sigma + \alpha e^\Delta - \alpha \varphi(\sigma)}. 
\end{equation}
Using the asymptotics for $\hat{n}_0(0) $ as well as the fact that $n_N (0) $ is given by \eqref{eq nk rig} we obtain that
\[
     \hat{n}_N (0) \sim \hat{n}_S (0) e^{- N E } \left( A(0) \left( 1- \frac{1}{\varphi_2 }\right) + (G- A(0) ) \left( 1- \frac{\varphi(\sigma ) }{ \varphi_2}\right) \varphi(\sigma)^N   \right) \text{ as } N \to \infty. 
\] 
The fact that 
\[
\frac{dM}{dt } =- \alpha e^\Delta n_{N}  - \mu M 
\]
implies that 
\begin{equation*} 
\hat{M}(0) =\frac{ 1 - \alpha e^\Delta \hat{n}_{N}(0)}{\mu } \sim \frac{1}{\mu}. 
\end{equation*}

As a consequence we have that $\hat{n}_S(0) \sim \overline n_S \hat{M}(0) $ as $N \to \infty $. 
This, combined with the fact that when $\Delta < \Delta_c $ we have that $\varphi(\sigma) < 1 $ while when $\Delta > \Delta_c $ we have that $\varphi(\sigma ) > 1 $ to obtain the desired result. 
\end{proof} 
\begin{theorem} \label{thm: rigorous discr}
    Let $p_{res} (\sigma) $ be the probability of response defined by \eqref{prob response}. 
    Then  we have that 
    \begin{itemize}
        \item if $\Delta < \Delta_c $, then 
        \[
        p_{res}(\sigma ) \sim \left(  1 + C_1 e^{N(E- b)} \right)^{-1}   \text{ as } N \to \infty
        \]
        where $C_1:=\left[ \alpha e^{\Delta} A(0) \overline n_S \left(1- \varphi(\sigma) e^{-(E + \Delta)} \right) \right]^{-1}$. 
        \item if instead $\Delta > \Delta_c $, then we have that
        \[
        p_{res} (\sigma )   \sim \left(  1 + C_2 e^{N(E- b - \log(\varphi(\sigma))} \right)^{-1}   \text{ as } N \to \infty
        \]
where 
\[
C_2=  \left(\overline{n}_S  \alpha e^\Delta  (G- A(0) ) \left( 1-(\varphi(\sigma ))^2 e^{- (\Delta + E)}\right) \right)^{-1} . 
        \]
    \end{itemize}
\end{theorem}
\begin{proof} 
Recall  that 
\begin{equation} \label{M} 
 \hat{M} (0)= \frac{1 - \alpha e^\Delta \hat{n}_{N} (0) }{\mu }.  
\end{equation}
Plugging $\hat{M}(0) $ in the asymptotics for $\hat{n}_N(0) $ and using the fact that $\frac{ 1}{\varphi_2} = \varphi(\sigma) e^{- (E+ \Delta)}$, we obtain the desired conclusions.  
\end{proof}

The constants $C_1$ and $C_2 $ in Theorem  \ref{thm: rigorous discr}  are the same constants that we obtained using matched asymptotics, i.e. \eqref{formal p res sub} and \eqref{formal p super}. (In order to see that the constant $C_2 $ is the same constant that appears in \eqref{formal p super} notice that $G-A(0)= B A(0) $ where $B$ is given by \eqref{B}). 
\begin{remark}
Notice that Theorem \ref{thm: rigorous discr} implies that the kinetic proofreading model has strong discrimination properties, in the sense of Definition \ref{def:strong discr}, when $\Delta > \Delta_c$. 
\end{remark}

\section{Some PDE limits} \label{sec:PDEs}
In this section we derive two PDEs that can be seen as an approximation of the discrete model \eqref{ODEs} as $N \to \infty$ in different parameters regimes. 
These  PDEs are particular types of renewal equations (see for instance \cite{perthame2007transport}) and can be analysed with computations that are simpler compared to the ones of Section \ref{sec:rigorous}.

When $e^\sigma \approx \frac{1}{N } $ and $\Delta > E $ and both $E$ and $\Delta $ are of order one, we will derive the following PDE
\begin{align} 
\partial_\tau f(\tau,x) &= - \beta  f(\tau,x) + \alpha \left( e^E- e^\Delta \right)  \partial_x f(\tau,x) - \delta f(\tau,x), \quad x \in (0,L) \label{PDE1} \\
f(\tau, 0) & = \frac{\beta }{\alpha ( e^\Delta - e^E ) }\int_0^L f(\tau, x) dx, \label{PDE1bound}
\end{align}
for a suitable $\beta >0$, $\delta >0$ and $L> 1$ and initial condition
\begin{equation} \label{IC PDE1}
\int_0^L f(0, x ) dx =1.  
\end{equation}

The steady states of the equation above behave as  $ e^{-\lambda x } $ where  $\lambda =  \frac{\beta + \delta }{ \alpha(e^\Delta - e^E) }$. Since $\lambda$ depends non trivially on $\beta$ the model has strong discrimination properties. Notice that this is expected as in this regime we have that 
\[
\Delta_c = \log \left( 1 + \frac{ \beta  }{ N \alpha (e^E-1) } \right) \sim \frac{ \beta  }{ N \alpha (e^E-1) } \text{ as } N \to \infty. 
\]

Instead when we assume that $e^\sigma \approx \frac{1}{N } $ and $\Delta \approx 1$ and  $E\approx \frac{1}{N } $, we derive the following PDE
\begin{align}
\partial_\tau f (\tau ,x) =& e^{-  x } m (\tau )  - \beta  f(\tau, x) - \alpha  (1-e^\Delta)  \partial_x f (\tau,x) - \delta f(\tau,x) \label{PDE2} \\ 
m(\tau ) =& \frac{\beta}{1-e^{-L}}  \int_0^L f(\tau, x) dx \label{PDE2m} \\
f(\tau, 0)=& 0. \label{PDE2bound}
\end{align}

If we ignore the loss term $- \delta f$, the steady states of this equation have the form  $  \overline m e^{- x } \left(  1+ C e^{ (1- \lambda ) x} \right) $ for a suitable constants $\overline  m $ and $C$ and for $\lambda =  \frac{\beta }{ \alpha (e^\Delta -1) }$. Notice that in order to have strong discrimination properties in this case we need to assume that 
\[
e^\Delta > 1+ \frac{ \beta  }{ \alpha }.
\]
This condition is expected, indeed in this regime we have that 
\[
e^{\Delta_c} = 1 + \frac{e^\sigma }{ \alpha (e^E -1 )} = 1 + \frac{\beta  }{ N \alpha (e^{1/N} -1 )}  \sim   1+ \frac{ \beta }{ \alpha } \text{ as } N \to \infty. 
\]

Notice that in order to obtain the PDEs above as an approximation of the discrete system \eqref{ODEs} we have to assume that $\Delta > E $. 
The two PDEs that we obtain have a transport term that has a sign, in particular in both the cases we have a transport of complexes toward larger phosphorylation states, i.e. to larger $k$. 
However we can have situations in which detailed balance fails with $\Delta > \Delta_c$, but in which the PDE approximation is not valid because $E > \Delta > \Delta_c$. In these cases, the discrete model \eqref{ODEs} cannot be approximated with a PDE and the transport towards small $k$ dominates the transport toward large $k$.
\subsection{PDE limit when $E\approx 1 $, $\Delta \approx 1$, $\Delta > E $, $e^\sigma \approx \frac{1}{N}$ and $\mu \approx \frac{1}{N}$} \label{sec:PDE1}
In this section we assume that the number of kinetic proofreading steps is $L N $ where $L > 1$ and where $N $ tends to infinity. 
In order to derive equation \eqref{PDE1} let us consider the behaviour of the solution to \eqref{ODEs} when $k \approx N $ and when $k \approx 1 $ under the assumption that $E\approx 1 $, $\Delta \approx 1$, $\Delta > E $ and $e^\sigma \approx \frac{1}{N}$, $\mu \approx \frac{1}{N}$. 
\paragraph{Behaviour of $n_k$ for $k \approx N $.}
Let us assume that $e^\sigma N = \beta>0 $ and that $\mu N=\delta >0$.
Notice that \eqref{ODEs} implies that 
\[
\partial_t n_k = n_S e^{- k E } - e^\sigma n_k + \alpha (e^E- e^\Delta) (n_{k-1} - n_k ) + \alpha e^\Delta (n_{k+1 } - 2 n_k + n_{k-1} )- \mu n_k 
\]
where 
\[ 
\partial_t n_S = - \left( 1 + e^{- E } \frac{1- e^{- N L E }}{1- e^{- E }} \right)  n_S + e^{\sigma }  \sum_{k=0}^{N L} n_k - \mu n_S 
\]
We now define $x= \frac{k}{N }$, $\tau = \frac{t }{ N } $, $f(\tau, x)= N n_k(t)$ and $m(\tau )= N n_S (t) $. 

We deduce that $f$ and $m $ satisfy the following system of equations 
\begin{align*}
\partial_\tau f (\tau ,x) =& N e^{- E x N } m (\tau )  - e^\sigma N f(\tau, x) + \alpha (e^E - e^\Delta ) N \left[ f\left(\tau, x- \frac{1}{N } \right)  - f(\tau, x) \right] \\
& + \alpha e^\Delta N \left[ f\left(\tau, x+\frac{1}{N } \right) - f(\tau, x) + f\left(\tau, x-\frac{1}{N } \right)  \right] - N \mu f(\tau,x) \\
\frac{1}{N} \partial_\tau m(\tau) =& e^\sigma \sum_{k=0}^{NL} f \left(\tau, \frac{k}{N} \right) -  \left( 1 + e^{- E } \frac{1- e^{- N L E }}{1- e^{- E }} \right)  m(\tau ) - \mu m(\tau). 
\end{align*}
Using the fact that $e^\sigma N = \beta $ and that $\mu N=\delta $ and the fact that $ N \left[ f\left(\tau, x- \frac{1}{N } \right)  - f(\tau, x) \right] \approx \partial_x  f(\tau,x) $ as $N \to \infty $ as well as the fact that 
$N \left[ f\left(\tau ,x+\frac{1}{N } \right) - f(\tau, x) + f\left(\tau, x-\frac{1}{N } \right)  \right] \approx \frac{1}{N} \partial^2_{xx}  f(\tau,x) \approx 0$ as $N \to \infty $ we deduce that 
\begin{align*}
\partial_\tau f (t,x) \approx & m(\tau) N e^{- E x N } - \beta f(\tau,x)  + \alpha (e^E - e^\Delta ) \partial_x  f(\tau,x) - \delta f(\tau,x), \ \text{ as } N \to \infty.
\end{align*}
Now notice that
\[
\frac{1}{N} \sum_{k=0}^{NL} f \left(\tau, \frac{k}{N} \right) \approx \int_0^L f(\tau,x) dx \text{ as } N \to \infty. 
\]
Therefore the equation for $m $ can be approximated as 
\[
 m(\tau) \approx \beta ( 1- e^{- E })  \int_0^L f(\tau, x ) dx   \text{ as } N \to \infty.
\] 
In particular this implies that $ m(\tau) N e^{ - Ex N } \rightarrow 0 $ as $N \to \infty$ for any $x>0$.
Therefore we obtain equation \eqref{PDE1}.

\paragraph{Behaviour of $n_k$ for $k \approx 1 $.}
In order to find the boundary condition at $0 $ for equation \eqref{PDE1} we have to match the behaviour obtained for $k \approx N $ with the behaviour in the region $k \approx 1 $. 
In this region we have that 
   \begin{align*}  
    \frac{d n_0 }{dt} & \approx  n_S  + \alpha e^E n_1 - \alpha e^{\Delta} n_0 \text{ as } N \to \infty  \\
        \frac{d n_k }{dt} & \approx n_S e^{- k E} - \alpha \left(  e^E + e^\Delta  \right)  n_k + \alpha e^E n_{k+1}  + \alpha e^{\Delta} n_{k-1} \text{ as } N \to \infty.
         \end{align*} 

As in Section \ref{sec:formal} we make the quasi steady state approximation, valid when $t \approx N $. 
As a consequence we deduce that in the region of $k \approx 1 $ we have that 
   \begin{align*}  
  0& \approx  n_S  + \alpha e^E n_1 - \alpha e^{\Delta} n_0 \text{ as } N \to \infty  \\
       0 & \approx n_S e^{- k E} - \alpha \left(  e^E + e^\Delta  \right)  n_k + \alpha e^E n_{k+1}  + \alpha e^{\Delta} n_{k-1} \text{ as } N \to \infty. 
         \end{align*} 
         With arguments that are analogous to the ones used in Section \ref{sec:formal} and Section \ref{sec:half line} we obtain that 
         \[
         n_k \approx B_1 n_S e^{- k E } + B_2(n_S), \ k \approx 1,  \text{ as } N \to \infty.
         \]
where 
\[
B_1 = \frac{-1}{ \alpha (e^E -1 ) (e^\Delta -1 )}, \quad B_2(n_S)=   \frac{  n_S(t) }{\alpha (1-e^{- E }) (e^\Delta - e^E)} . 
\]
We deduce that $n_k \approx B_2(n_S) $ as $k \to \infty $ (when $k \gg 1 $ and $N- k \gg 1 $), hence 
\[ 
  \lim_{k \to \infty } n_k (t)=  \frac{n_S(t) }{ \alpha (1-e^{- E })(e^\Delta - e^E) }  .
\] 
In order to match the region where $k \approx 1$ with the region where $k \approx N $ we need to have the matching condition 
\[
f(\tau, 0 ) = N  \lim_{k \to \infty } n_k. 
\]
The boundary condition \eqref{PDE1bound} follows. 

\subsection{PDE limit when $E\approx \frac{1}{N} $, $\Delta \approx 1$, $e^\sigma \approx \frac{1}{N}$ and $\mu \approx \frac{1}{N}$} \label{sec:PDE2}
As in Section \ref{sec:PDE1}, we assume that the number of kinetic proofreading steps is $L N $ where $L > 1$ and where $N $ tends to infinity. 
We consider the behaviour of $n_k$ when $k \approx N $ and when $k \approx 1 $ under the assumption that $E= \frac{1}{N}$, $\Delta \approx 1$, $e^\sigma= \frac{\beta}{N}$, $\mu= \frac{\delta}{N}$. 
\paragraph{Behaviour of $n_k$ for $k \approx N $.}
Defining $x= \frac{k}{N }$, $\tau = \frac{t }{ N } $, $f(\tau, x)= N n_k(t)$ and $m(\tau )= N^2 n_S (t) $, we deduce that $f$ and $m $ satisfy the following equations 
\begin{align*}
\partial_\tau f (\tau ,x) =& e^{- E x N } m (\tau )  - e^\sigma N f(\tau, x) + \alpha (e^E - e^\Delta ) N \left[ f\left(\tau, x- \frac{1}{N } \right)  - f(\tau, x) \right] \\
& + \alpha e^\Delta N \left[ f\left(\tau, x+\frac{1}{N } \right) - f(\tau, x) + f\left(\tau, x-\frac{1}{N } \right)  \right] - N \mu f(\tau,x) \\
\frac{1}{N} \partial_\tau m(\tau) =& e^\sigma N \sum_{k=0}^{NL} f \left(\tau, \frac{k}{N} \right) -  \left( 1 + e^{- E } \frac{1- e^{- N L E }}{1- e^{- E }} \right)  m(\tau ) - \mu m(\tau). 
\end{align*}
Taking $N \to \infty$ in the equation for $m $ we deduce that 
\[
m(\tau ) \approx \frac{\beta}{1-e^{-L}}  \int_0^L f(\tau, x) dx \text{ as } N \to \infty. 
\]
Moreover taking the limit as $N \to \infty $ in the equation for $f$ we deduce that \eqref{PDE2}-\eqref{PDE2m} holds. 
In order to obtain the boundary condition at $x=0$ we study the behaviour of $n_k $ when $k \approx 1$. 

\paragraph{Behaviour of $n_k$ for $k \approx 1 $.}
In this region, since $e^\sigma \approx \frac{1}{N}$, $E \approx \frac{1}{N}$ we have that 
   \begin{align*}  
    \frac{d n_0 }{dt} & \approx  \alpha  n_1 - \alpha e^{\Delta} n_0 \text{ as } N \to \infty \\
        \frac{d n_k }{dt} & \approx  - \alpha \left(  1 + e^\Delta  \right)  n_k + \alpha  n_{k+1}  + \alpha e^{\Delta} n_{k-1} \text{ as } N \to \infty. 
         \end{align*}
         Since $t \approx N $ we can assume that $\frac{d n_k }{dt} \approx 0$ and that $\frac{d n_0 }{dt} \approx 0$. 
         We deduce that 
            \begin{align*}  
0 & \approx n_S +  \alpha  n_1 - \alpha e^{\Delta} n_0 \text{ as } N \to \infty \\
0& \approx n_S - \alpha \left(  1 + e^\Delta  \right)  n_k + \alpha  n_{k+1}  + \alpha e^{\Delta} n_{k-1} \text{ as } N \to \infty 
         \end{align*}
This implies that 
\[
n_k (t) \approx \frac{ k n_S(t) }{\alpha (e^\Delta -1 )}    \text{ as } N - k \gg 1 \text{ and } k \gg 1. 
\]
The matching condition with the region where $k \approx N $ implies that 
\[
f(\tau, 0 ) = N  \lim_{k \to \infty } n_k \approx  \frac{ k  N n_S(t) }{\alpha (e^\Delta -1 )}\approx 0 \text{ as } N \to \infty, 
\] 
where we used the fact that $n_S \approx \frac{1}{N^2}$. 
This implies that \eqref{PDE2bound} holds. 

\section{Some models without detailed balance and lack of strong discrimination properties} \label{sec:alterantive models no db and no disc}
In this section we analyse some examples of kinetic proofreading networks that do not satisfy the property of detailed balance, and with the same amount of detailed balance $\Delta$ as the model analysed in Section \ref{sec:formal}, and that do not have strong discrimination properties.
The networks that we study have the same architecture of the linear kinetic proofreading network described in Section \ref{sec:model}.
We select the chemical rates in such a way that also for these networks equality \eqref{lack of db} holds for every cycle with $\Delta \neq 0$. In other words, in these networks the Wegscheider conditions associated to the cycles fails exactly in the same manner as in the network analysed in the previous sections. 
The only difference between the networks that we analyse in this section and the network analysed in the previous sections is the choice of the reaction in which the detailed balance property fails. In other words, the choice of the reaction whose rate is affected by the parameter $\Delta $. 

In Subsection \ref{sec:Delta to infty} we will also consider the case of $\Delta \to \infty $.
We show that if $\Delta \to \infty $, and $ \alpha \to 0 $ in such a way that $\alpha e^\Delta $ is of order one then we recover the one directional model analysed in \cite{franco2025stochastic}. 

\subsection{Detachment rate depending on $\Delta $}
We start by introducing a model in which the parameter $\Delta >0$ affects the detachment, i.e. the detachment reactions are accelerated with respect to the case in which $\Delta =0$. 
We show now that this model does not have strong discrimination properties. 

We consider the chemical reaction network generated by the following chemical reactions 
\begin{equation}\label{cycleDelta1}
S \overset{e^{- k E} }{\underset{e^{\sigma+ k \Delta }} \rightleftarrows} C_k  \overset{\alpha  }{\underset{\alpha e^E }\rightleftarrows } C_{k+1} \overset{e^{\sigma+ (1+k) \Delta }}{\underset{e^{- (k+1)E}} \rightleftarrows} S, \quad  k  \in \mathbb N. 
\end{equation}
Notice that for every $k $ the set of the reactions in \eqref{cycleDelta1} forms a cycle. Hence the circuit condition \eqref{weg condition} implies that the network satisfies the property of detailed balance if and only if $\Delta =0$. 

The system of equations associated with the network is the following 
\begin{align}\label{ODEsalt1}
    \frac{d n_S }{dt} &= - \frac{n_S}{1-e^{-E}} + e^{\sigma }  \sum_{k=0}^{\infty} n_k e^{k \Delta } \nonumber \\
    \frac{d n_0 }{dt} &= n_S  - e^{\sigma } n_0 + \alpha e^E n_1 - \alpha  n_0 \nonumber \\
        \frac{d n_k }{dt} &= n_S e^{- k E} - \left( e^{\sigma + k \Delta } + \alpha e^E + \alpha \right)  n_k + \alpha e^E n_{k+1}  + \alpha  n_{k-1} , \quad k \geq 1. 
\end{align}
We analyse the behaviour of the steady states $\tilde{n} =(\tilde{n}_S, \tilde{n}_0, \tilde{n}_1, \dots, \tilde{n}_N )$ of \eqref{ODEsalt1} as $k \to \infty $, more precisely we show that 
\begin{equation} \label{behaviour nk alt1}
\tilde{n}_k \sim C \tilde{n}_S e^{\sigma - k(E+\Delta)  } \  \text{ as } k \to \infty 
\end{equation}
for a suitable constant $C$. Notice that in this case the system does not have strong discrimination properties. 

We now indicate with heuristic arguments how the behaviour \eqref{behaviour nk alt1} can be obtained. First of all notice that we can always find a particular solution to 
\[ 
\tilde{n}^{(p)}_S e^{- k E} - \left( e^{\sigma + k \Delta } + \alpha e^E + \alpha  \right)  \tilde{n}_k^{(p)} + \alpha e^E \tilde{n}_{k+1}^{(p)}  + \alpha  \tilde{n}_{k-1}^{(p)} =0
\]
that is such that $\tilde{n}^{(p)}_S=1$ and $\tilde{n}_{0}^{(p)} = \tilde{n}_{1}^{(p)} =0$. 
We now study the homogeneous solutions corresponding to \eqref{ODEsalt1}, i.e. the solutions to the equation 
\begin{equation} \label{hom alt12}
0=  \alpha e^E \tilde{n}_{k+1}^{(h)}  + \alpha  \tilde{n}_{k-1}^{(h)} - \left( e^{\sigma + k \Delta } + \alpha e^E + \alpha  \right)  \tilde{n}_k^{(h)}, \quad k \geq 1. 
\end{equation}
We use the change of variables $\tilde{n}_k^{(h)} = e^{W_k} $ and deduce that 
\begin{equation} \label{hom alt1}
\alpha \left( e^{E+ W_{k+1} - W_k}  +e^{W_{k-1} - W_k}   \right) = e^{\sigma + k \Delta}+ \alpha + \alpha e^E, \quad k \geq 1.
\end{equation}
Notice that we have two possibilities, either $W_{k+1}- W_k \rightarrow \infty $ as $k \to \infty $, hence $ W_{k-1} - W_k\rightarrow -  \infty $ or $W_{k-1}- W_k \rightarrow \infty $ as $k \to \infty $, hence $ W_{k+1} - W_k\rightarrow -  \infty $. In the first case we have that $\alpha e^E e^{W_{k+1}- W_k } \sim e^{\sigma + k \Delta}$ as $k \to \infty $, therefore $W_k \sim \frac{k^2}{2} \Delta $ as $k \to \infty $. When instead we assume that $ W_{k-1} - W_k\rightarrow  \infty $ we deduce that $\alpha e^E e^{W_{k-1}- W_k } \sim e^{\sigma + k \Delta}$ as $k \to \infty $, hence $W_k \sim - \frac{ k^2}{2} \Delta$ as $k \to \infty $. 

Let $\psi_k^+ , \psi_k^-$ be two solutions to \eqref{hom alt12} such that $\psi_k^+ \sim e^{\frac{k^2 }{2} \Delta }$ as $k \to \infty $ and $\psi_k^- \sim e^{- \frac{k^2 }{2} \Delta }$ as $ k \to  \infty$.  
Let us define $\tilde{n}_k $ as 
\[ 
\tilde{n}_k = n_k^{(p)} \tilde{n}_S + a \psi_k^+ + b \psi_k^-, \quad k \geq 0. 
\]
Without loss of generality we can assume that $\psi_0^+= \psi_1^+=1$, so that $\psi^+_k$ is unique, and we select $a,b,\psi_0^- , \psi_1^+$ in such a way that 
\[
 a + b \psi_-^0 = \tilde{n}_0 = \frac{ \tilde{n}_S + \alpha e^E \tilde{n}_1 }{ \alpha + e^\sigma }  \text{ and } a +b\psi_-^1  =\hat{n}_1, 
\]
where 
\[
 \frac{ \tilde{n}_S  }{1-e^{-E}}  =  e^{\sigma }  \sum_{k=0}^{\infty} (n_k^{(p)} \tilde{n}_S + a \psi_k^+  )  e^{k \Delta } + b  e^{\sigma }  \sum_{k=0}^{\infty} \psi_k^-  e^{k \Delta }.
\]
This implies that $n_k^{(p)}\sim - a \frac{\psi_k^+}{\tilde{n}_S} $ as $k \to \infty $. More precisely, 
\[ 
n_k^{(p)} \tilde{n}_S + a \psi_k^+ \sim  e^{- \sigma - k (\Delta+ E) } \text{ as } k \to \infty , \quad \psi_k^- \sim e^{- \frac{k^2}{2 } \Delta  } \text{ as } k \to \infty
\]
Therefore \eqref{behaviour nk alt1} holds.

\subsection{Attachment rate depending on $\Delta $}

We consider a model in which the parameter $\Delta >0$ affects the attachment rate, hence the attachment reactions are accelerated. 
We show that this model can have strong discrimination properties under certain assumptions on the parameters. 

We study the chemical reaction network generated by the following chemical reactions 
\begin{equation}\label{cycleDelta1}
S \overset{e^{- k (E+\Delta) } }{\underset{e^{\sigma }} \rightleftarrows} C_k  \overset{\alpha  }{\underset{\alpha e^E }\rightleftarrows } C_{k+1} \overset{e^{\sigma  }}{\underset{e^{- (k+1)(E+\Delta) }} \rightleftarrows} S, \quad  k  \in \mathbb N. 
\end{equation}
The system of ODEs corresponding to this model is the following 
\begin{align}\label{ODEsalt2}
    \frac{d n_S }{dt} &= - \frac{n_S}{1-e^{-E}} + e^{\sigma }  \sum_{k=0}^{\infty} n_k e^{k (\Delta+E) } \nonumber \\
    \frac{d n_0 }{dt} &= n_S  - e^{\sigma } n_0 + \alpha e^E n_1 - \alpha  n_0 \nonumber \\
        \frac{d n_k }{dt} &= n_S e^{- k (E+\Delta) } - \left( e^{\sigma } + \alpha e^E + \alpha \right)  n_k + \alpha e^E n_{k+1}  + \alpha  n_{k-1} , \quad k \geq 1. 
\end{align}
We analyse the steady states $\tilde{n}=(\tilde{n}_S, \tilde{n}_k)_{k=0}^N $ of the system of ODEs \eqref{ODEsalt2}.
Using arguments that are analogues to the ones used in Section \ref{sec:formal} we deduce that the solutions of the form 
\[
\tilde{n}_k = B \tilde{n}_S e^{- k (E+\Delta) } + C g(\sigma)^k 
\]
where 
\[
g(\sigma) = \frac{1}{2 } \left( 1+ \frac{e^{\sigma-E}}{\alpha} + e^{-E} - \sqrt{ \left(1+ \frac{e^{\sigma-E}}{\alpha} +e^{-E } \right)^2 - 4 e^{- E }} \right) < 1 \] 
and where $B $ and $C$ are suitable constants. 
Moreover, the function $f$ defined as
\[
f(\Delta):= g(\sigma ) e^{ E+\Delta }  
\]
is such that $f(0)<1 $ while $\lim_{\Delta  \to \infty } f(\Delta ) >1$.
As a consequence there exists a $\Delta_c $ such that if $\Delta < \Delta_c $ then the model does not have strong discrimination properties while if $\Delta > \Delta_c $ the model has strong discrimination properties, i.e. the steady state is given  by $e^{- k \lambda(\sigma, E) } $ where $\lambda(\sigma, E) $ depends on $\sigma$. 

\subsection{Dephosphorylation rate depending on $\Delta $}
We assume now that $\Delta $ affects the dephosphorylation reaction. In this case we can obtain strong discrimination only if $\alpha $, the phosphorylation rate, is sufficiently large. 
Consider the chemical reaction network generated by the following chemical reactions 
\begin{equation}\label{cycleDelta1}
S \overset{e^{- k E} }{\underset{e^{\sigma }} \rightleftarrows} C_k  \overset{\alpha  }{\underset{\alpha e^{E-\Delta} }\rightleftarrows } C_{k+1} \overset{e^{\sigma  }}{\underset{e^{- (k+1)E }} \rightleftarrows} S, \quad  k  \in \mathbb N. 
\end{equation}
We formulate the system of equations 
\begin{align}\label{ODEsalt3}
    \frac{d n_S }{dt} &= - \frac{n_S}{1-e^{-E}} + e^{\sigma }  \sum_{k=0}^{\infty} n_k e^{k E } \nonumber \\
    \frac{d n_0 }{dt} &= n_S  - e^{\sigma } n_0 + \alpha e^{E- \Delta} n_1 - \alpha  n_0 \nonumber \\
        \frac{d n_k }{dt} &= n_S e^{- k E } - \left( e^{\sigma } + \alpha e^{E- \Delta} + \alpha \right)  n_k + \alpha e^{E- \Delta}  n_{k+1}  + \alpha  n_{k-1} , \quad k \geq 1. 
\end{align}
The steady states $\tilde{n}=(\tilde{n}_S, \tilde{n}_k)_{k=0}^N $ of the system of ODEs \eqref{ODEsalt3} are given by 
\[
\tilde{n}_k = B \tilde{n}_S e^{- k E } + C g(\sigma)^k 
\]
where $B, C$ are suitable constants depending on the parameters and where 
\[
g(\sigma) = \frac{1}{2 } \left[ 1+ e^{\Delta - E } + \frac{e^{\sigma + \Delta - E }}{ \alpha} - \sqrt{ \left( 1+ e^{\Delta - E } + \frac{e^{\sigma + \Delta - E }}{ \alpha} \right)^2  - 4 e^{\Delta - E }}  \right] < 1. 
\]
In this case we have that if $\frac{e^{\sigma- E }}{ \alpha (1- e^{- E } ) } <1 $, then there exists a critical $\Delta_c$  
\[
\Delta_c = \log\left(  \frac{1}{1-\frac{e^{\sigma- E }}{ \alpha (1- e^{- E } ) }  } \right)  >0 
\]
such that the asymptotic behaviour of $\tilde{n}_k$  as $k \to \infty $  depends on $\sigma $ only when $\Delta > \Delta_c$, i.e. $\tilde{n}_k \sim e^{ - k \lambda(\sigma, E) } $ as $k \to \infty $ where $\lambda (\sigma, E) $ depends on $\sigma $ only when $\Delta > \Delta_c$.

\subsection{$\Delta \to \infty $ and $\alpha e^{\Delta } \approx 1 $ }\label{sec:Delta to infty}
In this section we derive the one directional model that we analysed in \cite{franco2025stochastic} as a limit case of the model studied in this paper.

We consider the model with infinite states, i.e. we assume that the model is described by \eqref{ODE half line detail}.
We now study the steady states, that satisfy \eqref{ODE half line stationary}-\eqref{ODE half line stationary2}. As explained in Section \ref{sec:half line} the behaviour of the steady states is expected to be the following one 
\[ 
\tilde{n}_k \sim  e^{- k E } A(0) \overline n_S \left[ 1 + B \varphi(\sigma)^k \right], \text{ as } k \to \infty 
\]
where $B $ is given by \eqref{B} and where $\varphi(\sigma) $ is given by \eqref{varphi}. 
We assume that $\alpha e^\Delta =\gamma\approx 1 $, then 
\begin{align*}
\varphi(\sigma)&= \frac{1}{2 \alpha } \left( e^\sigma + \alpha (e^E+ e^\Delta) \right)  \left( 1 - \sqrt{ 1- \frac{4 \alpha^2 e^{\Delta + E }}{(e^\sigma + \alpha (e^E+ e^\Delta) )^2 } } \right) \\
& \sim   \frac{ \alpha e^{\Delta + E }}{e^\sigma + \alpha (e^E+ e^\Delta)  }  \sim e^E \frac{1}{1 + \frac{e^\sigma}{\gamma} }, \text{ as } \Delta \to \infty. 
\end{align*}

Since we have that 
\[ 
\tilde{n}_k \sim  e^{- k E } A(0) \overline n_S \left[ 1 + B \varphi(\sigma)^k \right], \text{ as } k \to \infty 
\]
we deduce that $\tilde{n}_k \sim e^{-\lambda(\sigma, E) k } $ where $\lambda(\sigma, E) = E $ if $\log\left(1+ \frac{e^\sigma}{\gamma } \right)> E  $ while $ \lambda (\sigma, E) =  \log \left(1+ \frac{e^\sigma}{\gamma} \right) $ when $\log\left(1+ \frac{e^\sigma}{\gamma} \right)  < E $. In the second case we have strong discrimination properties because $\lambda (\sigma , E) $ depends on $\sigma$.

In order to understand the condition for discrimination $\log\left(1+ \frac{e^\sigma}{\gamma} \right)  < E $, recall that 
\[ \Delta_c = \log \left(1+ \frac{e^\sigma }{ \alpha (e^E-1) } \right) .
\]
Assume now that $\Delta \to \infty $ and that $\Delta - \Delta_c = x >0 $. Under this assumption, we expect the limiting model to have strong discrimination properties. 
This assumption implies that the parameters satisfy the assumption $\log\left(1+ \frac{e^\sigma}{\gamma} \right)  < E $, indeed
\[ 
1 < e^{x } = e^{ - \Delta} + \frac{e^{\sigma-\Delta} }{ \alpha (e^E-1) } \sim   \frac{e^{\sigma} }{ \gamma (e^E-1) } \text{ as } \Delta \to \infty. 
\]

In the model that we consider in \cite{franco2025stochastic} we have that $E=\infty$. Indeed in that paper we assume that when the ligands attach to the receptors then they form a complex that has always state $0. $
In that case is always true that $ \log \left( 1+ \frac{e^\sigma}{\gamma} \right)  < E=  \infty$, hence we always have discrimination. 

\section{Conclusions and open problems} 
In this paper we prove that a minimal amount of lack of detailed balance is necessary in order to have strong discrimination properties. In other words we prove that the kinetic proofreading systems must be out of equilibrium systems. This absence of equilibrium can be measured in terms of the frozen concentrations of some substances, like ATP or ADP, that must be sufficiently far from equilibrium. 

Given that the lack of detailed balance is a property on the chemical rates of the reactions that belong to the cycles, one could think that similar amounts of lack of detailed balance would result in a similar behaviour of the network. We prove that this is not the case.  We give examples of models that have identical amount of lack of detailed balance  and that do not have the same discrimination properties.

We also have shown that, in some parameters regimes, the chemical network studied in this paper can be approximated by means of PDEs for which the mathematical analysis of the discrimination properties is more amenable. 

It could be interesting to study the behaviour of the network when the lack of detailed balance $\Delta_k$ in the phosphorylation reaction $(k) \rightarrow (k+1)$ depends on $k$. 
It would be relevant to understand if also in this case a critical amount of detailed balance is needed in order to have strong discrimination. 

We could also consider more complex models in which the phosphorylation reaction is modelled in more detail. In particular we can assume that phosphorylation is a chain of reactions that involves $ATP$, $ADP $ and phosphate groups, but also other substances that are out of equilibrium, for instance some enzymes. It would be relevant to understand if including these substances in the system, hence changing the topology of the network, can improve the discrimination properties of the chemical network. 

In this paper we analyse in detail the discrimination property of the kinetic proofreading mechanism and its relation with detailed balance.
Understanding if a biological function can be achieved by passive systems (with detailed balance) or if it requires active processes is a relevant question for several biological systems.
For instance, in \cite{franco2025adaptation} we analyse whether the property of adaptation, that takes place in many relevant signalling systems, can be achieved or not by equilibrium chemical reaction networks. 

\bigskip 

\bigskip 

\textbf{Acknowledgements} The authors gratefully acknowledge the support by the Deutsche Forschungsgemeinschaft (DFG) through the collaborative research centre "The mathematics of emerging effects" (CRC 1060, Project-ID 211504053) and Germany's Excellence Strategy EXC2047/1-390685813.
The funders had no role in study design, analysis, decision to publish, or preparation of the manuscript.

\bigskip 

\bigskip

\textbf{E. Franco}: Institute for Applied Mathematics, University of Bonn,

\hspace{0.5 cm} Endenicher Allee 60, D-53115 Bonn, Germany

\hspace{0.5 cm} E-mail: franco@iam.uni-bonn.de

\hspace{0.5 cm} ORCID 0000-0002-5311-2124

\bigskip 

\textbf{J. J. L. Vel\'azquez}: Institute for Applied Mathematics, University of Bonn,

\hspace{0.5 cm} Endenicher Allee 60, D-53115 Bonn, Germany

\hspace{0.5 cm} E-mail: velazquez@iam.uni-bonn.de

\bibliographystyle{siam}

\bibliography{References}

\end{document}